\documentclass[11pt]{amsart}
\usepackage{epsfig,overpic}
\usepackage[usenames,dvipsnames]{color}
\usepackage[colorlinks=true,linkcolor=blue,citecolor=BrickRed]{hyperref}

\newtheorem{thm}{Theorem}[section]
\newtheorem{lem}[thm]{Lemma}
\newtheorem{cor}[thm]{Corollary}
\newtheorem{prop}[thm]{Proposition}

\theoremstyle{definition}

\newtheorem{note}[thm]{Note}

\theoremstyle{remark}

\newcommand{\R}{\mathbf{R}}

\newcommand{\Z}{\mathbf{Z}}

\newcommand{\ol}[1]{{\overline #1}}

\newcommand{\Durer}{D\"{u}rer }
\newcommand{\Durers}{D\"{u}rer's }

\renewcommand{\S}{\mathbf{S}}

\renewcommand{\tilde}{\widetilde}

\DeclareMathOperator{\inte}{int}

\DeclareMathOperator{\st}{st}

\DeclareMathOperator{\cl}{cl}

\begin{document}

\title{Affine unfoldings of convex polyhedra}

\author[M.~Ghomi]{Mohammad Ghomi}
\address{School of Mathematics, Georgia Institute of Technology,
Atlanta, GA 30332}
\email{ghomi@math.gatech.edu}
\urladdr{www.math.gatech.edu/$\sim$ghomi}

\date{Last Typeset \today.}
\thanks{Research of the author was supported in part by NSF Grants DMS-0336455, DMS-1308777, and Simons Collaboration Grant 279374.}

\begin{abstract}
We show that every convex polyhedron admits a simple edge unfolding after an affine transformation. In particular there exists no combinatorial obstruction to a positive resolution of \Durers unfoldability problem, which answers a question of Croft, Falconer, and Guy. Among other techniques, the proof employs a topological characterization for embeddings among the planar immersions of  the disk.
\end{abstract}

\maketitle

\tableofcontents

\section{Introduction}\label{sec:intro}
A well-known problem  in geometry \cite{do:book, orourke:book, pak:book, ziegler:book},  which may be traced back to the Renaissance artist Albrecht D\"{u}rer \cite{durer}, is concerned with cutting  a convex polyhedral surface along some spanning tree of its edges so that it may be
isometrically embedded, or unfolded without overlaps,   into the plane. Here we show that this is always possible after an affine transformation of the surface. In particular, unfoldability of a convex polyhedron does not depend on its combinatorial structure, which settles a  problem of Croft, Falconer, and Guy \cite[B21]{cfg}. 

In this work a (compact) \emph{convex polyhedron} $P$  is  the boundary of the convex hull of a finite number of affinely independent points of Euclidean space $\R^3$.
A \emph{cut tree} $T\subset P$ is
a (polygonal) tree which  includes all the vertices of $P$, and each of its leaves is a vertex of $P$. Cutting $P$ along $T$ yields a compact surface $P_T$ which  admits an isometric immersion $P_T\to\R^2$ (see Section \ref{sec:trees}),  called an \emph{unfolding} of $P$. This unfolding is   \emph{simple}, or an \emph{embedding}, if it is one-to-one. 
We say $P$ is in \emph{general position} with respect to a unit vector or \emph{direction} $u$ provided that the \emph{height function} $h(\cdot):=\langle \cdot, u\rangle$ has a unique maximizer and a unique minimizer on  vertices of $P$. Then $T$ is \emph{monotone} with respect to $u$ provided that $h$ is (strictly) decreasing on every simple path in $T$ which connects a leaf of $T$ to the vertex minimizing $h$. For $\lambda>0$, we define the (normalized) \emph{affine stretching} parallel to $u$ as the linear transformation $A_\lambda\colon\R^3\to\R^3$ given by 
\begin{equation*}
A_\lambda(p):=\frac{1}{\lambda}\big(\, p+(\lambda-1) \langle p,u\rangle u\,\big),
\end{equation*}
and set $X^\lambda:=A_\lambda(X)$ for any  $X\subset\R^3$. Note that if $u=(0,0,1)$,  then $A_\lambda(x,y,z)=(x/\lambda,y/\lambda,z)$. 
Thus $A_\lambda$  makes any convex polyhedron  arbitrarily ``thin" or ``needle-shaped" for large $\lambda$. Our main result is as follows:

 \begin{thm}\label{thm:main}
Let $P$ be a convex polyhedron, $u$ be a direction with respect to which $P$ is in general position, and $T\subset P$ be a cut tree which  is monotone with respect to $u$. Then the unfolding of $P^\lambda$ generated by  $T^\lambda$  is simple for sufficiently large 
$\lambda$.
\end{thm}

  When a cut tree  is composed  of the edges of $P$,  or is a spanning tree of the edge graph of $P$, the corresponding unfolding  is called an \emph{edge unfolding}. If $P$ admits a simple edge unfolding, then we say $P$ is \emph{unfoldable}. Note that there exists an open and dense set of directions $u$ in the  sphere $\S^2$ with respect to which $P$ is in general position. Furthermore, it is easy to construct monotone spanning edge trees  for every such direction. They may be generated, for instance, via the well studied ``steepest edge" algorithm \cite{schlick,lucier,do:book}, or a general procedure described in Note \ref{note:monotone}. Thus  Theorem \ref{thm:main} quickly yields:

\begin{cor}\label{cor:main}
An affine stretching of a convex polyhedron, in almost any direction, is unfoldable. \qed
\end{cor} 

An example of this phenomenon  is illustrated in Figure \ref{fig:overlap}. 
\begin{figure}[h] 
   \centering
   \includegraphics[height=1.5in]{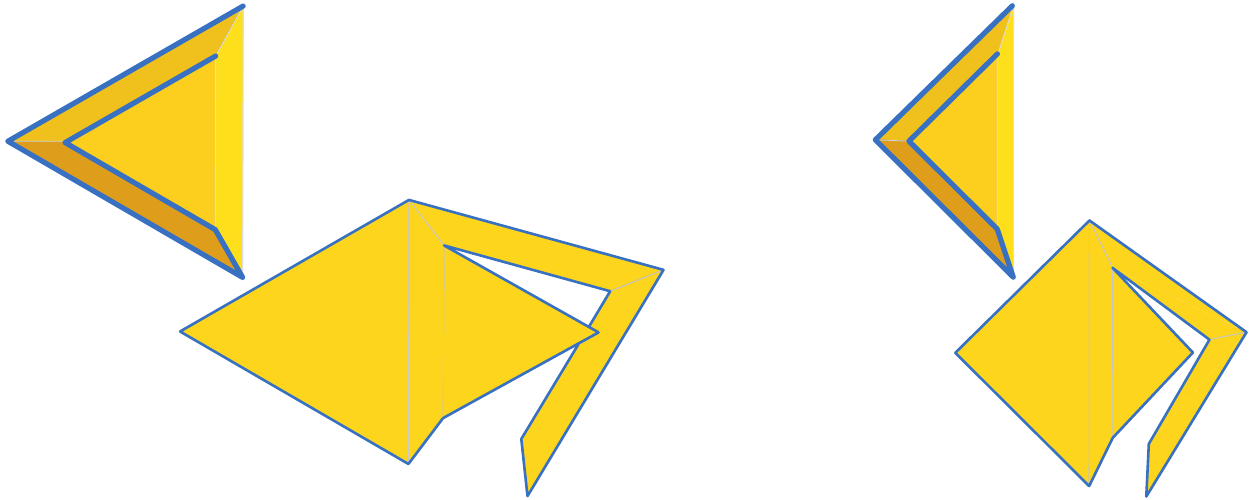} 
   \caption{}
   \label{fig:overlap}
\end{figure}
The left side of this figure shows a truncated tetrahedron (viewed from ``above") together with an overlapping  unfolding of it generated by a monotone edge tree.  As we see on the right side, however, the same edge tree generates a simple unfolding once the polyhedron has been stretched.

The rest of this work will be devoted to proving  Theorem \ref{thm:main}. We will start in Sections \ref{sec:paths} and \ref{sec:mixed} by recording some basic definitions and observations concerning the composition of paths in convex polyhedra and their developments in the plane. In particular we discuss the notion of ``mixed developments" which arises naturally in this context and constitutes a useful technical tool. Then, in Section \ref{sec:trees}, we will show that to each cut tree  there is associated a path whose development coincides with the boundary of the corresponding unfolding. Thus \Durers problem may be viewed as the search for spanning edge trees with simple developments. To this end, we will obtain in Section \ref{sec:immersion} a topological criterion for deciding when a closed planar curve which bounds an immersed disk is simple. This will be the principal tool for proving Theorem \ref{thm:main}, which will be utilized by 
means of an induction on the number of leaves of the  cut tree. To facilitate this approach we will study the structure of monotone cut trees  in Section \ref{sec:monotonetrees}, and the effect of affine stretchings on their developments in Section \ref{sec:affine}. Finally, these observations will be  synthesized  in Section \ref{sec:proof} to complete the proof.

The earliest known examples of simple edge unfoldings for convex polyhedra are due to 
\Durer \cite{durer}, although the  problem which bears his name  was first formulated by Shephard \cite{shephard}.  Furthermore, the assertion that a solution can always be found,  which has been dubbed \Durers conjecture, appears to have been first published by Gr\"{u}nbaum  \cite{grunbaum, grunbaum2}.
There is empirical evidence both for and against this  supposition. On the one hand, computers have found simple edge unfoldings for countless convex polyhedra through an exhaustive search of their spanning edge trees. On the other hand, there is still no algorithm for finding the right tree \cite{schlick,lucier}, and computer experiments suggest that the probability that a random edge unfolding of a generic polyhedron overlaps itself approaches $1$ as the number of vertices grow \cite{schevon:thesis}.  General cut trees  have been studied at least as far back as Alexandrov \cite{alexandrov:polyhedra} who first established the existence of simple unfoldings  (not necessarily simple \emph{edge} unfoldings) for all convex polyhedra, see also \cite{itoh&orourke,miller&pak,dd&orourke} for  recent related results. Other references and background  may be found in \cite{do:book}.

\begin{note}
A chief difficulty in assailing \Durers problem  is  the lack of any intrinsic characterization for an edge of a convex polyhedron $P$. Indeed the edge graph of $P$ is not the unique graph in $P$ whose vertices coincide with those of $P$, whose edges are geodesics, and whose faces are convex. It seems reasonable to expect that \Durers conjecture should be true if and only if it holds for this wider class of generalized edge graphs. This approach  has been studied by Tarasov \cite{tarasov}, who has announced some negative results in this direction.
\end{note}

\begin{note}
As we mentioned above, one way to generate some monotone trees in a convex polyhedron is via the ``the steepest edge" algorithm which has been well studied due to its relative effectiveness in finding simple unfoldings. Indeed Schlickenrieder \cite{schlick} 
 had conjectured 
that every convex polyhedron contains at least one steepest edge tree which generates a simple unfolding. He had successfully tested this conjecture in thousands of cases, after a thorough examination of various kinds of spanning edge trees and cataloguing their failure to produce simple unfoldings. Subsequently, however, Lucier \cite{lucier} produced a counterexample to Schlickenrieder's conjecture. Although it is not clear  whether all monotone trees in Lucier's example fail to produce  simple unfoldings.
\end{note}

\begin{note}
\Durers problem is usually  phrased in somewhat broader terms than described above: \emph{can every convex polyhedral surface be cut along some collection $T$ of its edges so that the resulting surface $P_T$ is connected and admits an isometric embedding into the plane?} In other words,  it is not a priori assumed that $T$ is a spanning tree. Assuming that this is the case, however, does not cause  loss of generality. Indeed it is obvious that the cut set $T$ must contain every vertex of $P$ (for otherwise $P_T$ will not be locally isometric to the plane), and $T$ may not contain any cycles (for then $P_T$ will not be connected). Furthermore, it follows  fairly quickly from the Gauss-Bonnet theorem  that $T$ must be connected \cite[Lem. 22.1.2]{do:book}. So $T$ is indeed  a spanning tree.
\end{note}

\begin{note}\label{note:monotone}
A general procedure for constructing monotone spanning edge trees $T$ in a convex polyhedron $P$ may be described as follows. The only requirement here is that $P$ be positioned so that it has a unique bottom vertex $r$. Then, since $P$ is convex, every vertex $v$ of $P$ other than $r$ will be adjacent to a vertex which lies  below it, i.e., has smaller $z$-coordinate. Thus, by moving down through a sequence of adjacent vertices, we  may connect $v$ to $r$ by means of a monotone edge path (with respect to $u=(0,0,1)$). Let $v_0$ be a top vertex of $P$, and $ B_0$ be a monotone edge path which connects $v_0$ to $r$. If $ B_0$ covers all vertices of $P$, then we set $T:= B_0$ and  we are done. Otherwise, from the remaining set of vertices choose an element $v_1$  which maximizes the $z$-coordinate on that set. Then we generate a monotone edge path $B_1$  by connecting $v_1$ to an adjacent vertex which lies below it and continue to go down through adjacent vertices until we reach a vertex of $B_0$ (including $r$). If $ B_0$ and $ B_1$ cover all the vertices of $P$, then we set $T:= B_0\cup B_1$ and we are done. Otherwise we repeat the above procedure, until  all vertices of $P$ have been covered.
\end{note}

\section{Preliminaries}\label{sec:paths}
For easy reference, we  begin by recording here the definitions and notation which will be used most frequently in the following pages.

\subsection{Basic terminology}\label{subsec:term}
Throughout this work $\R^n$ is the $n$-dimensional Euclidean space with standard inner product $\langle \cdot,\cdot\rangle$  and norm $\|\cdot\|$. Further $\S^{n-1}$ denotes the unit sphere in $\R^n$. The \emph{height function}   is the mapping $h\colon\R^n\to\R$ given by 
$
h(x_1,\dots, x_n):=x_n,
$
and $P$ denotes (the boundary of) a  (compact) convex polyhedron in $\R^3$ which is oriented by the outward unit normals to its faces.
  We assume that $P$   is positioned so that it has a  single \emph{top vertex} $\ell_0$ and a  single \emph{bottom vertex} $r$, i.e.,  $h$ has a unique maximizer and a unique minimizer on $P$.  Furthermore, $T$ is a cut tree of $P$  which is rooted at $r$. The \emph{leaves} of $T$ are the vertices of $T$ of degree $1$ which are different from $r$. The simple paths in $T$ which connect its leaves to $r$ will be called the \emph{branches} of $T$.
  We will assume, unless stated otherwise,  that $T$ is \emph{monotone}, by which we will always mean monotone with respect to  $u=(0,0,1)$. So $h$ will be (strictly) decreasing on each branch of $T$. We let $P_T$ be the surface obtained by cutting $P$ along $T$, and $\pi\colon P_T\to P$ be the corresponding projection (as will be defined in Section \ref{sec:trees}). Further $\ol P_T$ will denote the image of $P_T$ under an unfolding $P_T\to\R^2$. We say $\ol P_T$ is \emph{simple} if the unfolding map is one-to-one. More generally, for any mapping $f\colon X\to\R^2$ and subset $X_0\subset X$, we set $\ol X_0:=f(X_0)$ and say $\ol X_0$ is simple if
   $f$ is one-to-one on $X_0$. Finally,  by \emph{sufficiently large}  we mean for all values bigger than some constant.

\subsection{Paths and their compositions}\label{subsec:paths}
A line segment  in $\R^n$ is \emph{oriented} if one of its end points, say $a$, is designated as the ``initial point" and the other, say $b$, as the ``final point". Then the segment will be denoted by $ab$.
A \emph{path} $\Gamma$  is a sequence of oriented line segments in $\R^n$ such that the final point of each segment coincides with the initial point of the succeeding segment. These segments  are  called the \emph{edges} of $\Gamma$, and their end points constitute its \emph{vertices}. The vertices of $\Gamma$ inherit a natural ordering $\gamma_0,\dots,\gamma_k$, where $\gamma_0$ is the initial point of the first edge, $\gamma_k$ is the final point of the  last edge, and successive elements share a common edge. Conversely, any sequence of points $\gamma_0,\dots,\gamma_k$ of $\R^n$ with distinct successive elements determines a path denoted by:
$$
\Gamma=[\gamma_0,\dots,\gamma_{k}]:=(\gamma_0\gamma_1,\dots,\gamma_{k-1}\gamma_k).
$$
Then  $\gamma_0$, $\gamma_{k}$ are   the \emph{initial} and \emph{final} vertices of $\Gamma$ respectively. Any other vertex of $\Gamma$ will be called an \emph{interior vertex}. If only consecutive edges of $\Gamma$ intersect, and do so only at their common vertex, then  $\Gamma$ is \emph{simple}.
We say that $\Gamma$ is \emph{closed} if $\gamma_{k}=\gamma_0$, in which case we set $\gamma_{i+k}:=\gamma_{k}$,  and consider all vertices of $\Gamma$ to be interior vertices. An interior vertex is \emph{simple} if its adjacent vertices are distinct.  The \emph{trace} of $\Gamma$ is  the union of the edges of $\Gamma$, which will again be denoted by $\Gamma$. 
  For any pair of vertices $v$, $w$ of $\Gamma$ we let $vw=(vw)_\Gamma$ denote the subpath of $\Gamma$ with initial point $v$ and final point $w$. The trace of this path will also be denoted by $vw$.

 We utilize two different notions for combining a pair of paths $\Gamma=[\gamma_0,\dots,\gamma_k]$ and $\Omega=[\omega_0,\dots,\omega_\ell]$, when
$\gamma_k=\omega_0$. The \emph{concatenation} of these paths is given by
$$
\Gamma\bullet\Omega:=[\gamma_0,\dots,\gamma_k, \omega_1,\dots,\omega_\ell],
$$
while their \emph{composition} is defined as
$$
\Gamma\circ\Omega:=[\gamma_0,\dots,\gamma_{k-m}, \omega_{m+1},\dots,\omega_\ell],
$$
where $m$ is the largest integer such that $\gamma_{k-i}=\omega_i$ for $0\leq i\leq m$. One may think of $\Gamma\circ\Omega$ as the path obtained from $\Gamma\bullet\Omega$ by excising its largest subpath centered at $\gamma_k$ which double backs on itself, see Figure \ref{fig:compositions}. This notion has also been studied in \cite[p. 1770]{berestovski&plaut}.
Finally we set
$
\Gamma^{-1}:=[\gamma_k,\dots,\gamma_0].
$
Note that $\Gamma\circ\Gamma^{-1}=[\gamma_0]$ which  may be considered a \emph{trivial path}.

\begin{figure}[h] 
   \centering
   \begin{overpic}[height=0.7in]{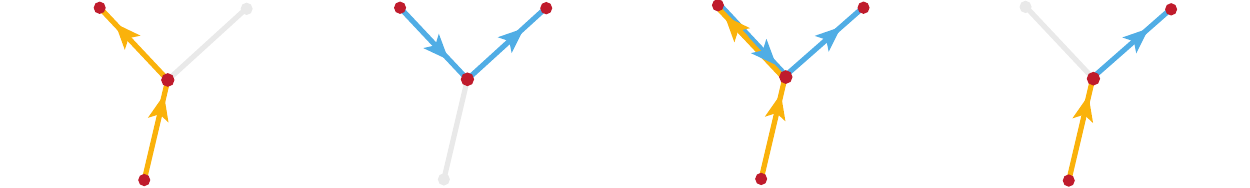} 
   \put(8,1){$\Gamma$}
    \put(32,1){$\Omega$}
     \put(52,1){$\Gamma\bullet\Omega$}
      \put(77,1){$\Gamma\circ\Omega$}
   \end{overpic}
   \caption{}
   \label{fig:compositions}
\end{figure}

\subsection{Sides and angles}\label{subsec:sidesAndangles}
In this section $P$ needs not be compact; in particular, it may stand for $\R^2\simeq\R^2\times\{0\}\subset\R^3$ with ``outward normal" $(0,0,1)$.
 A \emph{side} of a simple closed path $\Gamma$ in $P$ is the closure of a component of $P-\Gamma$. We may distinguish these  sides  as follows. Choose a point $x$ in the interior of an edge $\gamma_i\gamma_{i+1}$ of $\Gamma$, pick a side $S$ of $\Gamma$, let $F$ be the face of $S$ which contains $x$,  $n$ be the outward unit normal to $F$, and $\nu$ be a unit  normal to $\gamma_i\gamma_{i+1}$ which points inside $S$. Then  $S$ lies to the \emph{left} of $\Gamma$  provided that 
 $(\gamma_{i+1}-\gamma_i, n,\nu)$ has positive determinant; otherwise,  $S$ lies to the \emph{right} of $\Gamma$.
If $\Gamma$ is not simple or closed, one may still define a local notion of sides near its interior vertices as we describe below.

For any point $o\in P$, let $\st_o$ denote the \emph{star} of $o$, i.e., the union of  faces of $P$ which contain $o$. Orient the boundary curve $\partial st_o$ by choosing a cyclical ordering for its vertices, so that $st_o$ lies to the left of it. Any point $x\in \st_o-\{o\}$ generates a ray $R_x\subset\R^3$ which emanates from $o$ and passes through $x$. Let $\widehat\st_o$ denote the intersection of these rays with the unit ball in $\R^3$ centered at $o$. Then the \emph{total angle} of $P$ at $o$, denoted by $\angle_P(o)$, is  the length of  $\partial\widehat\st_o$. 
Next, for any pair of points $a$, $b$ in $\st_o-\{o\}$, we define the (left) \emph{angle} $\angle(a,o,b)$  of the path $[a,o,b]$. Consider the projection $\st_o-\{o\}\to\partial \widehat\st_o$ given by $x\mapsto\widehat x:=R_x\cap \partial \widehat\st_o$.
This establishes  a bijection $\partial st_o\to \partial \widehat \st_o$ which orients $\partial \widehat \st_o$. 
Let $|\cdot |$ denote the length of oriented segments of $\partial \widehat \st_o$ and set
\begin{equation*}
\angle(a,o,b):=
\begin{cases}
|\widehat b\widehat a|, & \widehat a\neq \widehat b;\\
\angle_P(o), &\widehat a=\widehat b.
\end{cases}
\end{equation*} 
 In particular note that if $\widehat a\neq\widehat b$, then 
\begin{equation}\label{eq:p1}
\angle(a,o,b)+\angle(b,o,a)=\angle_P(o).
\end{equation}
 If $\widehat a=\widehat b$, then we define the entire $\st_o$ as the \emph{left side} of $[a,o,b]$. Otherwise, $R_a\cup R_b$   divides $\st_o$ into a pair of components. The closure of each of these components will be called a \emph{side} of  $[a,o,b]$. The projection $\st_o-\{o\}\to\partial \widehat{st}_o$ maps one of these regions to the (oriented) segment $\widehat a\widehat b$  and the other to  $\widehat b\widehat a$, which will be called the \emph{right} and \emph{left} sides of $[a,o,b]$ respectively. Finally,   $c\in\st_o$ lies \emph{strictly} to the left (resp. right) of $[a,o,b]$ if $c$ lies in the left (resp. right) side of $[a,o,b]$ and is disjoint from  $R_a\cup R_b$. 
The following elementary observations will  be  useful throughout this work.

\begin{lem}\label{lem:abc}
Let $o\in P$, and $a$, $b$, $c\in \st_o-\{o\}$. Then we have:
\begin{enumerate}
\item[(i)]{$c$ lies strictly to the left of $[a,o,b]$ if and only if 
$
\angle (a,o,c)<\angle (a,o,b).
$
}
\item[(ii)]{If $c$ lies strictly to the left of $[a,o,b]$, then 
$
\angle (a,o,c)+\angle (c,o,b)=\angle (a,o,b).
$
}
\end{enumerate}
\end{lem}
\begin{proof}
To see (i)  first assume that $\widehat a=\widehat b$, see the left diagram in Figure \ref{fig:angle}. Then  $c$ lies strictly to the left of $[a,o,b]$ if and only if $\widehat c\neq\widehat a,\widehat b$.
Furthermore $\angle (a,o,c)<\angle_P(o)=\angle (a,o,b)$ if and only if $\widehat c\neq\widehat a,\widehat b$.
Next we establish (i) when $\widehat a\neq \widehat b$, see the right diagram in Figure \ref{fig:angle}. In this case, if $c$ lies strictly to the left of $[a,o,b]$, then
$\widehat c\in\inte(\widehat b\widehat a)$, the interior of $\widehat b\widehat a$ in $\partial \widehat{st}_o$. Thus
 \begin{figure}[h] 
   \centering
   \begin{overpic}[height=0.9in]{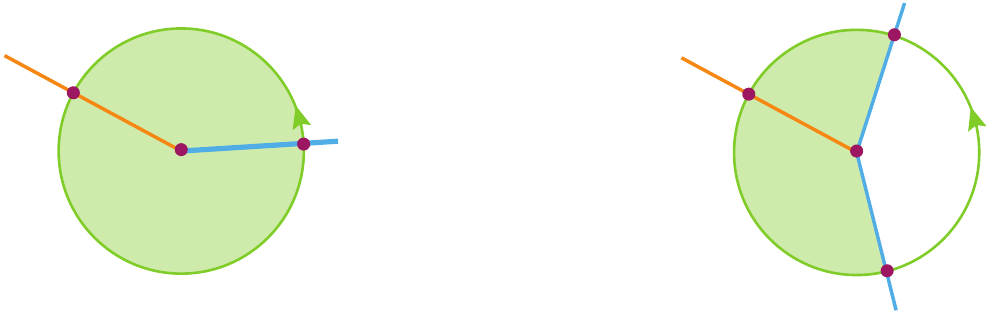} 
   \put(91,1){$\widehat a$}
   \put(92,28){$\widehat b$}
   \put(82,14){$o$}
    \put(70,19){$\widehat c$}
     \put(14.5,14){$o$}
    \put(3,19){$\widehat c$}
     \put(32,18.5){$\widehat a$}
   \put(32,11){$\widehat b$}
   \end{overpic}
   \caption{}
   \label{fig:angle}
\end{figure}
 $\angle (a,o,c)=|\widehat c\widehat a|<|\widehat b\widehat a|=\angle (a,o,b)$. Conversely, if $\angle (a,o,c)<\angle (a,o,b)$ (and $\widehat a\neq \widehat b$), then $\widehat c\neq \widehat a, \widehat b$. Consequently  $|\widehat c\widehat a|<|\widehat b\widehat a|$ which  yields $\widehat c\in \widehat b\widehat a$. So, since $\widehat c\neq \widehat a, \widehat b$,  it follows that  $c$ lies strictly to the left of $[a,o,b]$. 
To see (ii) note that if $\widehat a=\widehat b$ and $\widehat c\neq \widehat a, \widehat b$, then $\angle (a,o,c)+\angle (c,o,b)= \angle (a,o,c)+\angle (c,o,a)=\angle_P(o)=\angle(a,o,b)$.  If, on the other hand, $\widehat a\neq \widehat b$, and $c$ lies strictly to left of $[a,o,b]$, then $\widehat c\in\inte(\widehat b\widehat a)$. Thus  $\angle(a,o,b)=|\widehat b \widehat a|<|\widehat b \widehat c|+|\widehat c\widehat a|=\angle (a,o,c)+\angle (c,o,b)$.
\end{proof}

\section{Mixed Developments of Paths}\label{sec:mixed}
In this section we describe a general notion for  developing a path $\Gamma=[\gamma_0,\dots,\gamma_k]$ of $P$ into the plane, and show (Proposition \ref{prop:mixed})
 how this concept interacts with that of  composition of paths discussed in the last section. 
 First  we  define 
the \emph{left  angle}  of $\Gamma$ at an interior vertex $\gamma_i$  
 by 
 $$
 \theta_i=\theta_i[\Gamma]=\theta_{\gamma_i}[\Gamma]:=\angle (\gamma_{i-1},\gamma_i,\gamma_{i+1}).
 $$
 Further  the corresponding \emph{right angle}  is given by 
$$
{\theta_i}':=\angle (\gamma_{i+1},\gamma_i,\gamma_{i-1})=\angle_P(\gamma_i)-\theta_i,
$$
where the last equality follows from \eqref{eq:p1}. In particular we have 
\begin{equation}\label{eq:angles}
\theta_i+{\theta_i}'=\angle_P(\gamma_i)\leq 2\pi,
\end{equation}
due to the convexity of $P$.
It will also be useful to note that
$
{\theta_i}'[\Gamma]=\theta_{k-i}[\Gamma^{-1}].
$
A \emph{mixed development} of $\Gamma$  is a path
$
\ol\Gamma=[\ol\gamma_0,\dots, \ol\gamma_k]
$
in $\R^2$ with left angles $\ol\theta_i$,  and right  angles  $\ol{\theta_i}'$,  such that 
\begin{itemize}
\item[(i)]{
$\|\gamma_i-\gamma_{i-1}\|=\|\ol \gamma_i-\ol \gamma_{i-1}\|, \; \text{for}\; 1\leq i\leq k$};
 \item[(ii)] { $\ol\theta_i=\theta_i$ or $\ol{\theta_i}'={\theta_i}',\; \text{for}\; 1\leq i\leq k-1$}.
 \end{itemize}
  If  $\ol\theta_i=\theta_i$ for all $i$, then  $\ol\Gamma$ is a \emph{(left) development}. We say $\ol\Gamma$ is a \emph{mixed development} based at  an interior vertex $\gamma_\ell$ if $\ol{\theta_i}'={\theta_i}'$ for  $i\leq\ell$, and $\ol\theta_i=\theta_i$ for all $i>\ell$. This path will be denoted by $(\ol\Gamma)_{\gamma_\ell}$, and unless noted otherwise, the term $\ol\Gamma$ will be reserved to indicate a (left) development. We also set $(\ol\Gamma)_{\gamma_0}:=\ol\Gamma$. Note that $\ol\Gamma$ is uniquely determined once its \emph{initial condition} $(\ol\gamma_0,\ol u_0)\in\R^2\times\S^1$ has been prescribed, where  $\ol u_0:=(\ol\gamma_1-\ol\gamma_0)/\|\ol\gamma_1-\ol\gamma_0\|$ is the direction of the first edge.
A pair of paths $\Gamma$, $\Omega$ in $\R^2$ are \emph{congruent} if they coincide up to a (proper) rigid motion, in which case we write $\Gamma\equiv\Omega$.

\begin{figure}[h] 
   \centering
   \begin{overpic}[height=1.1in]{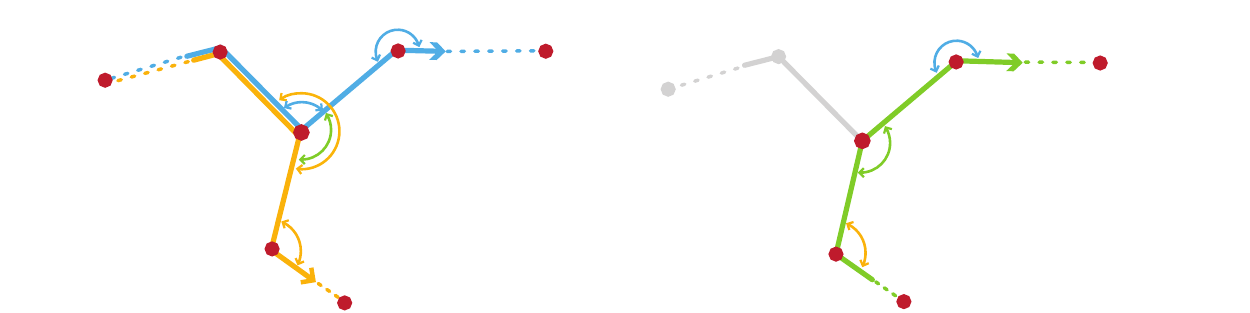} 
   \put(5,16){$\gamma_0$}
   \put(5,21){$\omega_0$}
    \put(19,14){$\gamma_m$}
   \put(22,19){$\omega_m$}
    \put(29,0){$\gamma_k$}
   \put(43,23){$\omega_\ell$}
     \put(72,12){$\delta_{k-m}$}
    \put(74,0){$\delta_0$}
   \put(84,22){$\delta_{k+\ell-2m}$}
   \end{overpic}
   \caption{}
   \label{fig:benz}
\end{figure}

\begin{prop}\label{prop:mixed}
Let $\Gamma=[\gamma_0,\dots,\gamma_k]$, $\Omega=[\omega_0,\dots,\omega_\ell]$ be a pair of paths in $P$ such that $\gamma_i=\omega_i$ for $i=0,\dots, m<\ell$. Further suppose that either $m=k$, or else
$\omega_{m+1}$ lies strictly to the left of $[\gamma_{m-1},\gamma_m, \gamma_{m+1}]$.   Then 
$$
(\overline{\Gamma})^{-1}\circ\overline{\Omega}\equiv\overline{(\Gamma^{-1}\circ\Omega)}_{\gamma_m},
$$
provided that  $\ol\Gamma$ and $\ol\Omega$ have the same initial conditions.
\end{prop}
\begin{proof}
Let $\Delta:=\Gamma^{-1}\circ\Omega$ and $\tilde\Delta:=(\ol\Gamma)^{-1}\circ\ol\Omega$. Then
\begin{align*}
\Delta&=[\gamma_k,\dots,\gamma_0]\circ[\omega_0,\dots,\omega_\ell]=[\gamma_k, \dots,\gamma_m,\omega_{m+1},\dots,\omega_\ell],\\
\tilde\Delta&=[\ol\gamma_k,\dots,\ol\gamma_0]\circ[\ol\omega_0,\dots,\ol\omega_\ell]=[\ol\gamma_k, \dots,\ol\gamma_m,\ol\omega_{m+1},\dots,\ol\omega_\ell].
\end{align*}
In particular note that $\Delta$, $\tilde \Delta$ each have $n:=k+\ell-2m+1$ vertices. If we denote these vertices by $\delta_i$, $\tilde\delta_i$, where $0\leq i\leq n-1$, then we have
$$
\delta_i=
\begin{cases}
\gamma_{k-i}, & i\leq k-m,\\
\omega_{i-k+2m}, & i\geq k-m;
\end{cases}
\quad\quad
\tilde\delta_i=
\begin{cases}
\ol\gamma_{k-i}, & i\leq k-m,\\
\ol\omega_{i-k+2m}, & i\geq k-m;
\end{cases}
$$
see Figure \ref{fig:benz}. 
In particular, $\gamma_m=\delta_{k-m}$. So we have to show that
$\tilde\Delta\equiv(\ol\Delta)_{\delta_{k-m}},$
 which means we need to  check:

 \begin{itemize}
\item[(i)]{$\|\delta_i-\delta_{i-1}\|=\|\tilde\delta_i-\tilde\delta_{i-1}\|,$} for $1\leq i\leq n-1$;
\item[(ii)]{${\theta_i}'[\tilde\Delta]={\theta_i}'[\Delta]$ for $1\leq i\leq k-m$, and $\theta_i[\tilde\Delta]=\theta_i[\Delta]$ for $k-m<i<n-1$}.
\end{itemize}
To establish  (i) note that, for $1\leq i\leq k-m$, 
$$
\|\delta_i-\delta_{i-1}\|=\|\gamma_{k-i}-\gamma_{k-i+1}\| =   \|\ol\gamma_{k-i}-\ol\gamma_{k-i+1}\| = \|\tilde\delta_i-\tilde\delta_{i-1}\|.
$$
Furthermore, for $k-m\leq i\leq n-1$,
$$
\|\delta_i-\delta_{i-1}\|= \|\omega_{i-k+2m}-\omega_{i-k+2m-1}\| = \|\ol\omega_{i-k+2m}-\ol\omega_{i-k+2m-1}\| = \|\tilde\delta_{i}-\tilde\delta_{i-1}\|.
$$
Next we check  (ii). For $1\leq i<k-m$,
$$
{\theta_i}'[\Delta]= {\theta_i}'[\Gamma^{-1}]=\theta_{k-i}[\Gamma]=\theta_{k-i}[\ol\Gamma]={\theta_i}'[(\ol\Gamma)^{-1}]={\theta_i}'[\tilde\Delta].
$$
Furthermore, for $k-m<i<n-1$,
$$
\theta_i[\Delta]=
\theta_{i-k+2m}[\Omega]=\theta_{i-k+2m}[\ol\Omega]=\theta_i[\tilde\Delta].
$$
It remains to check  that ${\theta'_{k-m}}[\Delta]={\theta'_{k-m}}[\tilde\Delta]$, and to this end it suffices to show:
\begin{equation}\label{eq:mixed1}
\theta_{k-m}'[\Delta]=\theta_m[\Gamma]-\theta_m[\Omega]\quad\text{and}\quad \theta_{k-m}'[\tilde\Delta]=\theta_m[\ol\Gamma]-\theta_m[\ol\Omega].
\end{equation}
To establish the first equation in \eqref{eq:mixed1} note that
\begin{align*}
\theta_m[\Omega]+\theta'_{k-m}[\Delta]&= \angle(\omega_{m-1},\omega_m,\omega_{m+1})+\angle(\delta_{k-m+1},\delta_{k-m},\delta_{k-m-1})\\
&=\angle(\gamma_{m-1},\gamma_m,\omega_{m+1})+\angle(\omega_{m+1},\gamma_m,\gamma_{m+1}).
\end{align*}
Further, since $\omega_{m+1}$ lies strictly on the left of $[\gamma_{m-1},\gamma_m,\gamma_{m+1}]$,  Lemma \ref{lem:abc}(ii) yields
$$
\angle(\gamma_{m-1},\gamma_m,\omega_{m+1})+\angle(\omega_{m+1},\gamma_m,\gamma_{m+1})=\angle(\gamma_{m-1},\gamma_m,\gamma_{m+1})=\theta_m[\Gamma].
$$
 The second equation in \eqref{eq:mixed1} follows from a  similar calculation, once we check that $\ol\omega_{m+1}$ lies strictly to the left of $[\ol\gamma_{m-1},\ol\gamma_m,\ol\gamma_{m+1}]$. Indeed, since $\omega_{m+1}$ lies strictly on left of $[\gamma_{m-1},\gamma_m,\gamma_{m+1}]$, Lemma \ref{lem:abc}(i) yields 
$$
\theta_m[\Omega]=\angle(\omega_{m-1},\omega_m,\omega_{m+1})=\angle(\gamma_{m-1},\gamma_{m},\omega_{m+1})<\angle(\gamma_{m-1},\gamma_{m},\gamma_{m+1})=\theta_m[\Gamma].
$$
So $\theta_m[\ol\Omega]=\theta_m[\Omega]< \theta_m[\Gamma]=\theta_m[\ol\Gamma]$. Consequently, 
$$
\angle(\ol\gamma_{m-1},\ol\gamma_{m},\ol\omega_{m+1})=\angle(\ol\omega_{m-1},\ol\omega_m,\ol\omega_{m+1})=\theta_m[\ol\Omega]<\theta_m[\ol\Gamma]=\angle(\ol\gamma_{m-1},\ol\gamma_{m},\ol\gamma_{m+1}).
$$
So, by Lemma \ref{lem:abc}(i),  $\ol\omega_{m+1}$ lies strictly to the left of $[\ol\gamma_{m-1},\ol\gamma_m,\ol\gamma_{m+1}]$ as claimed.
\end{proof}

\section{The Tracing Path   of a Cut Tree}\label{sec:trees}
Here we describe  precisely  how  a cut tree $T$ determines an unfolding of $P$.  Further we show that the boundary of this unfolding coincides with a development of a certain path $\Gamma_T$ which traces $T$. This leads to the main result of this section, Proposition \ref{prop:disk} below, which shows that an unfolding of $P$ generated by $T$ is simple if and only if the development of  $\Gamma_T$  is simple.
We start by recording some basic lemmas. In this section $T$ needs not be monotone.

Since leaves of $T$ are vertices of $P$, $T$ partitions each face  of $P$ into a finite number of polygons.
Let $F_T(P):=\{F_i\}$ be the disjoint union of these polygons, and $\pi\colon F_T(P)\to P$ be the projection  generated by the  inclusion maps $F_i\hookrightarrow P$. Glue each pair $F_i$, $F_j\in F_T(P)$ along a pair  $E_{in}$, $E_{jm}$ of their edges if and only if $\pi(E_{in})$, $\pi(E_{jm})\not\in T$ and $\pi(E_{in})=\pi(E_{jm})$. This yields a compact surface $P_T$ (which we may think of as having resulted from ``cutting" $P$ along $T$).
The inclusion maps $F_i\hookrightarrow P$ again define a natural projection $\pi\colon P_T\to P$, which is the identity map on  $\inte(P_T):=P_T-\partial P_T=P-T$. So, since $T$ is contractible,  $P_T$  is a topological disk. 
Also note that $P_T$ inherits an orientation from $P$, which in turn induces  a cyclical ordering $\tilde v_0,\dots, \tilde v_n$ on the vertices of $\partial P_T$ so that $P_T$ lies   to the left of $\partial P_T$, i.e., every $\tilde v_i$ has an open neighborhood $U_i$ in $P_T$ such that $\pi(U_i)$ lies to the left of $[\pi(\tilde v_{i-1}), \pi(\tilde v_{i}), \pi(\tilde v_{i+1})]$ in $P$.
Since $P_T$ contains no vertices in its interior, and all the interior angles of $\partial P_T$  are  less than $2\pi$,  it is locally isometric to the plane. Therefore, since $P_T$ is simply connected, it  may be isometrically immersed  in the plane, e.g.,
 see \cite[Lem. 2.2]{ghomi:verticesA}. An  \emph{immersion} is  a locally one-to-one continuous map, and is \emph{isometric} if it  preserves distances. So we have established:

\begin{lem}\label{lem:flattenning}
$P_T$ is simply connected and is locally isometric to $\R^2$. In particular, there exists an isometric immersion $P_T\to\R^2$.\qed
\end{lem}

Any such immersion will be called an \emph{unfolding} of $P$ (generated by $T$) 
provided that it is \emph{orientation preserving}, i.e., $\ol P_T$ lies locally on the left  of $\overline{\partial P}_T$, with respect to the orientation that $\overline{\partial P}_T$ inherits from $\partial P_T$. Recall that for any set $X\subset P_T$, we let $\ol X$ denote the image of $X$ under the unfolding $P_T\to\R^2$, and say  $\ol X$ is simple provided that $X\to\R^2$ is one-to-one. 

\begin{lem}\label{lem:disk1}
$\ol P_T$ is simple if and only if $\overline{\partial P}_T$ is simple. 
\end{lem}
\begin{proof}
This is a special case of the following general fact, see \cite{ghomi:topology}:
if $M$ is a connected compact surface with boundary components $\partial M_i$, and $M\to\R^2$ is an immersion,
then $\ol M$ is simple if and only if each $\overline{\partial M}_i$ is simple.
\end{proof}

So, as far as \Durers problem is concerned, we just need to decide when $\overline{\partial P}_T$ is simple. To this end it would be useful to think of 
$\overline{\partial P}_T$ not as the restriction of the unfolding of $P_T$ to $\partial P_T$, but rather as the development of a path  of $P$. This path is given by
$$
\Gamma_T=[v_0,\dots,v_n]:=[\pi(\tilde v_0),\dots,\pi(\tilde v_n)],
$$
 where $\tilde v_0,\dots, \tilde v_n$ is the cyclical ordering of the vertices of $\partial P_T$ mentioned above. Thus $\Gamma_T$  traces $\pi(\partial P_T)=T$, and $\pi$ establishes a bijection $v_i\leftrightarrow\tilde v_i$ between the vertices of $\Gamma_T$ and $\partial P_T$. For each $\tilde v_i$  let $\tilde{st}_{\tilde v_i}$ denote the star of $\tilde v_i$ in $P_T$. Then,
since $P_T$ lies to the left of $\partial P_T$, it follows that 
\begin{lem}\label{lem:stv}
For every vertex $v_i$ of $\Gamma_T$, the left side of $[v_{i-1},v_i,v_{i+1}]$ in $P$ coincides with $\pi(\tilde{st}_{\tilde v_i})$.\qed
\end{lem}
\noindent In particular, the left angles of $\Gamma_T$ in $P$ are the same as the interior angles of $P_T$. So, since the unfolding $P_T\to\R^2$ is orientation preserving, it follows that the left angles of
 $\overline{\partial P}_T$ are the same as those of  a development $\ol\Gamma_T$ of $\Gamma_T$. 
 Thus  $\overline{\partial P}_T$ is congruent to $\ol\Gamma_T$, and Lemma \ref{lem:disk1} yields:

 \begin{prop}\label{prop:disk}
$\overline P_T$ is simple if and only if  $\ol\Gamma_T$ is simple.\qed
 \end{prop}
 
\pagebreak 
\section{Criteria for Embeddedness of Immersed Disks}\label{sec:immersion}

As we discussed in the last section, an immersed disk in the plane is simple (or embedded) if and only if its boundary is simple. Here we generalize that observation. Let $D\subset\R^2$ be the unit disk centered at the origin, with oriented boundary $\partial D$.   For   $p$, $q\in\partial D$,   let $pq\subset \partial D$ denote the segment with initial point $p$ and final point $q$. 
Recall that an \emph{immersion}  is   a continuous locally one-to-one map. 
Further  recall that for any $X\subset D$, and mapping $f\colon D\to\R^2$, we set $\ol X:=f(X)$, and say $\ol X$ is simple if $f$ is one-to-one on $X$.
\begin{figure}[h] 
   \centering
   \begin{overpic}[height=0.75in]{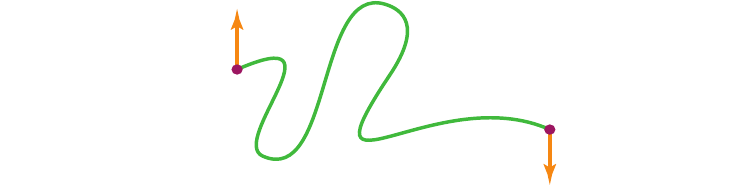} 
       \end{overpic}
   \caption{}
   \label{fig:weakmonotone}
\end{figure}
A simple curve segment in $\R^2$,  whose end points do not have the same height, is \emph{weakly monotone} (with respect to the direction $(0,1)$) if it may be extended  to an unbounded simple curve by attaching a vertical ray to its top end point which extends upward, and a vertical ray to its bottom end point which extends downward, see Figure \ref{fig:weakmonotone}.
The main result of this section is that an immersed disk is embedded whenever its boundary admits a decomposition into weakly monotone paths:

\begin{prop}\label{cor:disk}
Let $D\overset{f}{\to}\R^2$ be an immersion with polygonal boundary. Suppose there is a pair of points $p_0$, $p_1$ in $\partial D$ such that $\overline{p_0p_1}$ and $\overline{p_1p_0}$ are  weakly monotone. Then $\ol D$ is simple.
\end{prop}

The basic strategy for proving the above proposition is to extend $f$ to an immersion of a larger disk which has simple boundary and thus is one-to-one. To this end first note that by \emph{polygonal boundary} here we mean that there are points $v_i$, $i\in \Z_k$, cyclically arranged along $\partial D$ so that $f$ maps each oriented segment $v_iv_{i+1}$ to a line segment. Then we obtain a closed polygonal path $[\ol v_0,\dots, \ol v_{k}]$ in $\R^2$. Since $f$ is locally one-to-one, each $v_i$ has a neighborhood $U_i$ in $D$ such that $\ol U_i$ lies one side of $[\ol v_{i-1},\ol v_{i}, \ol v_{i+1}]$ as defined in Section \ref{subsec:sidesAndangles}, and it is easy to see that this side must be the same for all $i$. Thus we may say that $\ol D$ \emph{lies locally on one side} of $\overline{\partial D}$.

\begin{proof}[Proof of Proposition \ref{cor:disk}]
As discussed above,  $\ol D$ lies locally one one side of the path $\overline{\partial D}$. We may assume that this is the left hand side after composing $f$ with a reflection of $\R^2$.
Since $\overline{p_0p_1}$, $\overline{p_1p_0}$ are  weakly monotone, they are simple and their end points have different heights. 
Suppose that $\ol p_0$ is the end point with the lower height, and let $R_0$ be the vertical ray which emanates from $\ol p_0$ and extends downward. Similarly, let $R_1$ be the vertical ray which emanates from $\ol p_1$ and extends upward. Let $C$ be a circle so large which contains $\ol D$ in the interior of the region it bounds. Then $R_0$, $R_1$ intersect $C$ at precisely one point each, say at $x_0$ and $x_1$ respectively, see Figure \ref{fig:circles}.
\begin{figure}[h] 
   \centering
   \begin{overpic}[height=1.25in]{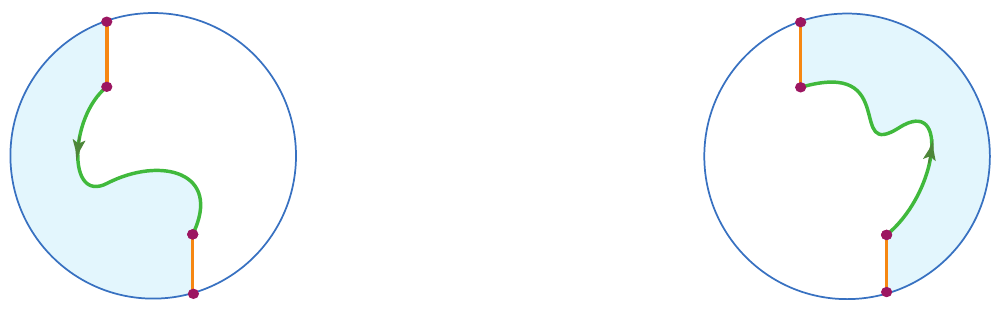} 
   \put(18,-1){\small $x_0$}
   \put(20.5,7){\small $\ol p_0$}
   \put(12,22){\small $\ol p_1$}
   \put(10,31){\small $x_1$}
   \put(7,8){\small $D_1$}
   \put(88,-1){\small $x_0$}
   \put(84.5,7){\small $\ol p_0$}
   \put(76,22){\small $\ol p_1$}
   \put(79,31){\small $x_1$}
   \put(88,23){\small $D_0$}
         \end{overpic}
   \caption{}
   \label{fig:circles}
\end{figure}
 Now consider the oriented composite path $x_0x_1:=x_0\ol p_0\cup \overline{p_0p_1}\cup \ol{p_1} x_1$ shown on the right diagram in Figure \ref{fig:circles}. Since $\overline{p_0p_1}$ is weakly monotone, $x_0 x_1$ is simple, and thus it divides the region 
bounded by $C$ into a pair of disks. Let $D_0$ be the disk which lies to the right of $x_0x_1$. Similarly, let $x_1 x_0:=x_1\ol p_1\cup \overline{p_1p_0}\cup \ol{p_0} x_0$, and $D_1$ be the disk which lies on the right of the oriented path $x_1 x_0$, as shown in the left diagram in Figure \ref{fig:circles}.
Now glueing $D_0$ and $D_1$ along the segments $x_0\ol{p_0}$ and $x_1\ol{p_1}$ yields an immersed annulus $A$. Note that by construction $A$ lies locally on the right of $\overline{\partial D}$. Thus, gluing $A$ to $\overline{D}$ along $\overline{\partial D}$ yields an immersed disk, say $D'$. Note that $\partial D'=C$ which is simple. Thus it follows (via \cite[Lem. 1.1]{ghomi:topology}, or Lemma \ref{lem:preX}) that $D'$ and consequently $\ol D$ is simple as claimed.
\end{proof}

The criteria proved above were the precise conditions we need in the proof of Theorem \ref{thm:main}. See Appendix A for more general criteria concerned with embeddedness of immersed disks.

\section{Structure of Monotone Cut Trees}\label{sec:monotonetrees}
 Here we describe how the leaves of $T$ inherit a cyclical ordering from $P$, which in turn orders the branches of $T$. This  will  be used to define a sequence of paths $\Gamma_i$ in $T$, together with a class of related paths $\Gamma_i'$. The paths $\Gamma_i$ join the top and bottom vertices of $P$, while $\Gamma_i'$ are closed paths which  correspond to the boundary of certain disks $D_i\subset P_T$.

\subsection{Leaves $\ell_i$ and junctures $j_i$}
Let $\Gamma:=\Gamma_T$ be the  path tracing $T$ defined in Section \ref{sec:trees}. 
Note that  each edge $E$ of $T$ appears precisely twice in $\Gamma$,  because there are precisely two faces  $F_1$, $F_2$ of $P_T$ such that $\pi(F_1)$ and $\pi(F_2)$ are adjacent to $E$.
This quickly yields:

\begin{lem}\label{lem:degree}
Let $v$ be a vertex of $T$ which has degree $n$ in $T$. Then there are precisely $n$ vertices of $\Gamma$ which coincide with $v$.\qed
\end{lem}

In particular each leaf of $T$ occurs only once in $\Gamma$. Consequently,
$\Gamma$ determines a unique ordering 
 $\ell_0$, $\ell_1,\dots, \ell_{k-1}$ of the leaves of $T$. Further we set $\ell_{i+k}:=\ell_i$, and designate $\ell_0$ (the top vertex of $P$) as the initial vertex of $\Gamma$. 
  Recall that a vertex of  $\Gamma$ is simple if its adjacent vertices are distinct.

\begin{lem}\label{lem:mi0}
Any vertex of $\Gamma$ which is not a leaf or root of $T$ is simple.
\end{lem}
\begin{proof}
Let $v_i$ be a vertex of $\Gamma$ which is not simple. We will show that the degree of $v_i$ in $T$ is $1$, which is all we need.
Since $v_i$ is non-simple,  the left side of $[v_{i-1},v_i,v_{i+1}]$ is the entire star $\st_{v_i}$ by definition. Thus, by Lemma \ref{lem:stv},  $\pi(\tilde\st_{\tilde v_i})=\st_{v_i}$. Choose $r>0$ so small that the metric ``circle" $C\subset P_T$ of radius $r$ centered at
$\tilde v_i$ lies in the interior of  $\tilde\st_{\tilde v_i}$. Then $\pi(C)\subset\inte(\st_{v_i})$ is a simple closed curve enclosing $v_i$ which intersects $T$ only once.
 Thus $\deg_T(v_i)=1$  as claimed.
\end{proof}

 Using the last lemma, we now show that the leaves of $T$ may be characterized via the height function $h$ as follows:

\begin{lem}\label{lem:PT}
A vertex of $\Gamma$ is a local maximum point of  $h$ on the sequence of vertices of $\Gamma$ if and only if it is a leaf of $T$. 
\end{lem}
\begin{proof}
Suppose that $v$ is a leaf of $T$, and let $u$, $w$ be its adjacent vertices in $\Gamma$. Since $T$ is monotone, and $v\neq r$, there exists a vertex $v'$ of $T$ which is adjacent to $v$ and lies below it. Since $\deg_T(v)=1$,  $u=v'=w$. In particular, $u$, $w$ lie below $v$. So $v$ is a local maximizer of $h$. 
Conversely, suppose that $v$ is a local maximizer of $h$. Then $u$, $w$  lie below $v$, because $T$ does not have horizontal edges. Let  $\Omega_u$, $\Omega_w$ be the simple monotone paths in $T$ which connect $u$, $w$  to $r$ respectively. Then $vu\bullet\Omega_u$ and $vw\bullet\Omega_w$ are  simple, since they are monotone.  Hence $u=w$, by the uniqueness of simple paths in $T$. So $v$ is not  simple and therefore must be either a leaf or the root of $T$, by Lemma \ref{lem:mi0}. The latter is impossible, since $v$ is a local maximizer of $h$. 
\end{proof}

  \begin{figure}[h] 
   \centering
   \begin{overpic}[height=1.5in]{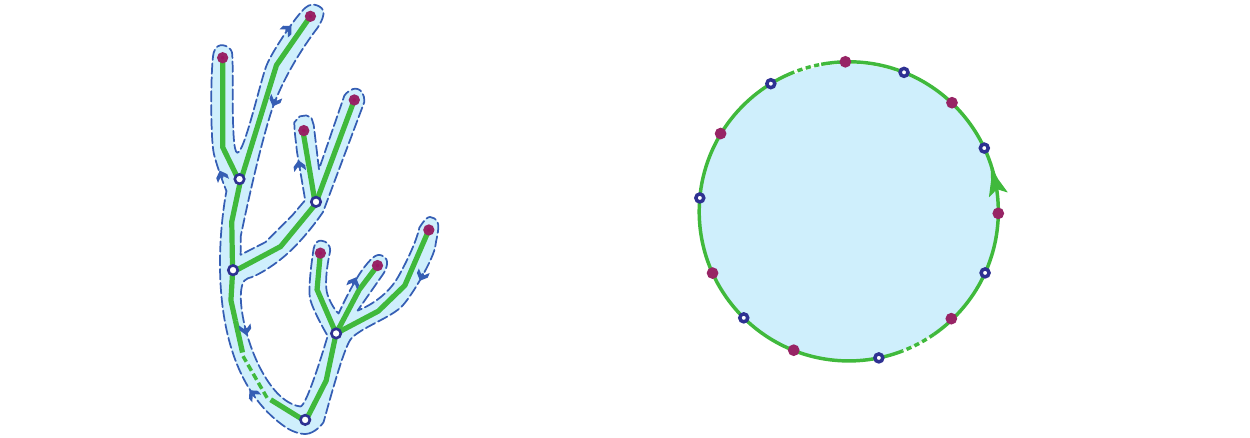} 
      \put(26.5,33.5){$\ell_0$}  
      \put(23.5,27){$\ell_1$}
      \put(29,16){$\ell_i$}
      \put(32,19){$\ell_{i+1}$}
      \put(14,13){$j_0$}
      \put(66,17){$P_T$}
      \put(81,17){$\tilde\ell_0$} 
      \put(80,23){$\tilde j_0$} 
      \put(77,27){$\tilde\ell_1$} 
       \put(54,12){$\tilde\ell_i$}
         \put(56,7){$\tilde j_i$}
          \put(28,5){$j_i$}
      \put(60,4){$\tilde\ell_{i+1}$}
      \put(16,1){$\Gamma$}       
   \end{overpic}
   \caption{}
   \label{fig:cut}
\end{figure}

 It follows from Lemma \ref{lem:PT} that between every pair of consecutive leaves $\ell_i$, $\ell_{i+1}$ of $\Gamma$ there exists a unique vertex $j_i$, called a \emph{juncture}, which is a local minimizer of $h$, see Figure \ref{fig:cut}. 
Note  that some  junctures  of $\Gamma$ may coincide with each other, or with the root $r$ of $T$. For any ordered pair $(v,w)$ of vertices of $T$ let $(vw)_T$ be the (unique) simple path in $T$ joining $v$ to $w$. Note that the paths $\ell_j j_i$ and $j_i\ell_{i+1}$ of $\Gamma$ are monotone and therefore simple. Thus
\begin{equation}\label{eq:ellem}
 (\ell_i j_i)_\Gamma=(\ell_i j_i)_T, \quad\text{and}\quad  (j_i\ell_{i+1})_\Gamma=(j_i\ell_{i+1})_T.
\end{equation}

\subsection{Branches $\beta_i$ and the paths $\Gamma_i$}
For $0\leq i\leq k-1$, we define the \emph{branches}  of $T$ as the paths
$
\beta_i:=(\ell_i r)_T,
$
which connect each leaf of $T$ to its root. 
Note that, by  \eqref{eq:ellem}, we have
\begin{equation}\label{eq:bi}
\beta_i=(\ell_i r)_T=(\ell_i j_i)_T\bullet (j_i r)_T=(\ell_i j_i)_\Gamma\bullet (j_i r)_T.
\end{equation}
Having ordered the branches of $T$, we  now describe the first class of paths which are useful for our study of monotone trees: 
\begin{equation}\label{eq:gammai}
\Gamma_i:=(\ell_0\ell_i)_\Gamma\bullet \beta_i=(\ell_0j_i)_\Gamma\bullet (j_ir)_T,
\end{equation}
for $0\leq i\leq k-1$.
 See Figure \ref{fig:subtreesIII} for some examples. Next we record how the composition of these paths is related to the branches of $T$.
\begin{figure}[h] 
   \centering
   \begin{overpic}[height=1.5in]{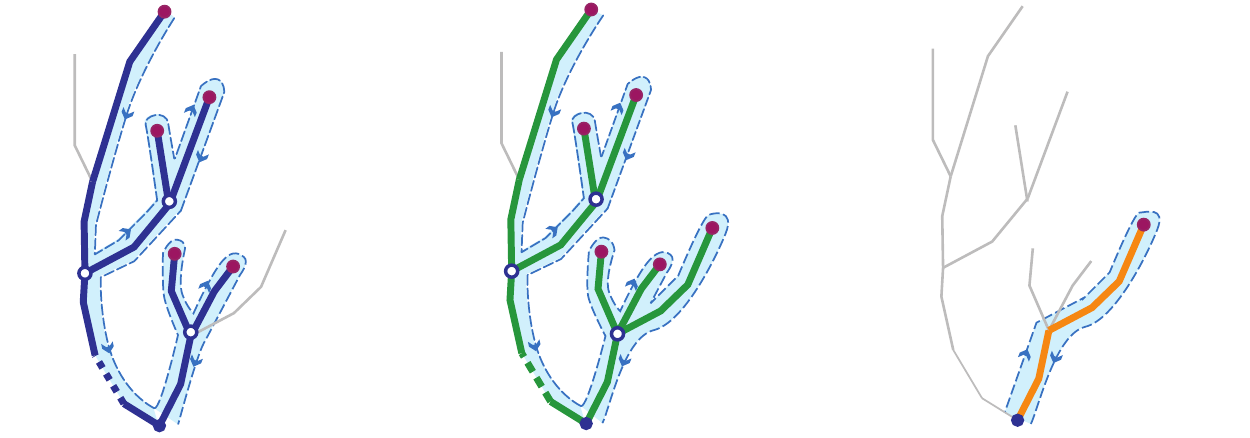} 
    \put(5,0){$\Gamma_i$}
     \put(37,0){$\Gamma_{i+1}$}
     \put(65,0){$\Gamma_{i}^{-1}\circ\Gamma_{i+1}$}
      \put(18,16){$\ell_i$}
       \put(56,19){$\ell_{i+1}$}
        \put(91,19){$\ell_{i+1}$}
   \end{overpic}
   \caption{}
   \label{fig:subtreesIII}
\end{figure}

\begin{lem}\label{lem:branches}
For $0\leq i\leq k-2$, 
$
\Gamma_i^{-1}\circ\Gamma_{i+1}=\beta_{i+1}^{-1}\bullet \beta_{i+1}.
$
\end{lem}
\begin{proof}
By \eqref{eq:ellem}, \eqref{eq:bi} and \eqref{eq:gammai}
\begin{align*}
\Gamma_i^{-1}\circ\Gamma_{i+1}
&=\big((\ell_0 j_i)_\Gamma\bullet (j_i r)_T\big)^{-1}\circ \big((\ell_0\ell_{i+1})_\Gamma\bullet \beta_{i+1}\big)\\
&= (rj_i)_T\bullet (j_i\ell_0)_{\Gamma^{-1}}\circ (\ell_0j_i)_\Gamma\bullet(j_i\ell_{i+1})_\Gamma\bullet \beta_{i+1}\\
&= (rj_i)_T\bullet (j_i\ell_{i+1})_T\bullet \beta_{i+1}\\
&= (r\ell_{i+1})_T\bullet \beta_{i+1}\\
&= \beta_{i+1}^{-1}\bullet \beta_{i+1}.
\qedhere
\end{align*}
\end{proof}

The following observation shows, via Lemma \ref{lem:abc}(i), that $\Gamma_{i+1}$ lies to left of $\Gamma_i$ near  $j_i$, if $j_i$ is an interior vertex of $\Gamma_i$.

\begin{lem}\label{lem:mi}
If $j_i\neq r$, then
$
\theta_{j_i}[\Gamma_{i+1}]<\theta_{j_i}[\Gamma_{i}],
$
for $0\leq i\leq k-2$.
\end{lem}
\begin{proof}
Let $v$, $w$ be the vertices of $\Gamma_{i+1}$ which  precede and succeed $j_i$ respectively. By Lemma \ref{lem:mi0}, $v\neq w$. 
Next let $u$ denote the vertex of $\Gamma_i$ which succeeds $j_i$, see Figure \ref{fig:mi}. 
We need to show that
 $
 \angle(v,j_i,w)<\angle(v,j_i,u).
 $ 
Suppose, towards a contradiction, that $\angle(v,j_i,w)\geq\angle(v,j_i,u)$. The equality in the last inequality cannot occur, because by 
\begin{figure}[h] 
   \centering
   \begin{overpic}[height=0.9in]{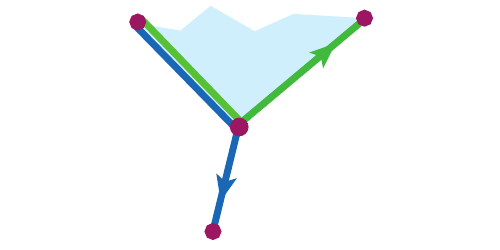} 
   \put(35,3){$u$}
    \put(20,43){$v$}
     \put(76,43){$w$}
      \put(38,22){$j_i$}
       \put(47,35){S}
       \put(63,29){$\Gamma_{i+1}$}
       \put(48,9){$\Gamma_{i}$}
   \end{overpic}
   \caption{}
   \label{fig:mi}
\end{figure}
definitions of $\Gamma_i$ and $\Gamma_{i+1}$, $u$ lies below $j_i$ while $w$ lies above it (so $u\neq w$). Thus we may assume that $\angle(v,j_i,w)>\angle(v,j_i,u)$.
Then, by Lemma \ref{lem:abc}(i), $u$ lies strictly in the left side $S$ of $[v,j_i,w]$. Consequently  $uj_i$   intersects the interior  of $S$, which means $\inte(S)\cap T\neq\emptyset$. But this is  impossible because
$S=\pi(\widetilde{st}_{\tilde j_i})$ by Lemma \ref{lem:stv} which yields
\begin{equation}\label{eq:intS}
\inte(S)=\inte\big(\pi(\widetilde{st}_{\tilde j_i})\big)=\pi\big(\inte(\widetilde{st}_{\tilde j_i})\big)\subset \pi\big(\inte(P_T)\big)= P-T.
\qedhere
\end{equation}
\end{proof}

Now we are ready to prove the main result of this subsection:

\begin{prop}\label{cor:branches}
For $0\leq i\leq k-2$, 
$
(\overline{\Gamma_i})^{-1}\circ\ol\Gamma_{i+1}\equiv(\overline{\beta_{i+1}^{-1}\bullet \beta_{i+1}})_{j_i}.
$
\end{prop}
\begin{proof}
By Proposition \ref{lem:branches}, $\beta_{i+1}^{-1}\bullet \beta_{i+1}=\Gamma_i^{-1}\circ\Gamma_{i+1}$. So we just need to check that 
$
(\overline{\Gamma_i^{-1}\circ\Gamma_{i+1}})_{j_i}\equiv (\overline{\Gamma_i})^{-1}\circ\ol\Gamma_{i+1},
$
which follows from Proposition  \ref{prop:mixed} via Lemma \ref{lem:mi}. More specifically, there are two cases to consider. If $j_i=r$, then 
$\Gamma_i$ is a subpath of $\Gamma_{i+1}$, which corresponds to the case ``$m=k$" in Proposition \ref{prop:mixed}. If $j_i\neq r$, then Lemma \ref{lem:mi} together with Lemma \ref{lem:abc}(i) ensure that $\Gamma_{i+1}$ lies to the left of $\Gamma_i$ near $j_i$, and  so the hypothesis of Proposition \ref{prop:mixed} is again satisfied.
\end{proof}

\subsection{Dual branches $\beta_i'$ and the paths $\Gamma_i'$} 
To describe the second class of paths which we may associate to a monotone tree,  we first establish the existence of a collection of paths $\beta_i'$ which are in a sense dual to the branches $\beta_i$ defined  above.

\begin{prop}\label{prop:dualbranch}
Each leaf $\ell_i$ of a monotone cut tree $T$ may be connected to the top leaf $\ell_0$ of $T$ via a monotone path $\beta_i'$ in $P$, which intersects  $T$ only at its end points. 
\end{prop}

Assume for now that the above proposition holds. Then
 for each leaf $\ell_i$, we fix a path 
$\beta_i'$ given by this  proposition and set
\begin{equation}\label{eq:defGamPrime}
\Gamma_i':=
\begin{cases}
(\ell_0\ell_i)_\Gamma\bullet \beta_i', & 1\leq i\leq k-1;\\
\Gamma, & i=k.
\end{cases}
\end{equation}
Note that since the interior of $\beta_i'$ lies in $P-T$, it lifts to a unique path $\tilde\beta_i'$ in $P_T$,  see Figure \ref{fig:pathII}, such that $\pi(\tilde\beta_i')=\beta_i'$. Consequently each $\Gamma_i'$ corresponds to a simple closed curve $\tilde\Gamma_i'$ in $P_T$ where 
$\tilde\Gamma_i':=(\tilde\ell_0\tilde\ell_i)_{\partial P_T}\bullet \tilde\beta_i'$, for $1\leq i\leq k-1$, and $\tilde\Gamma_k':=\partial P_T$. Let $D_i\subset P_T$ be the disk bounded by $\tilde\Gamma_i'$ which lies to the left of it, and note that 
\begin{figure}[h] 
   \centering
   \begin{overpic}[height=1.4in]{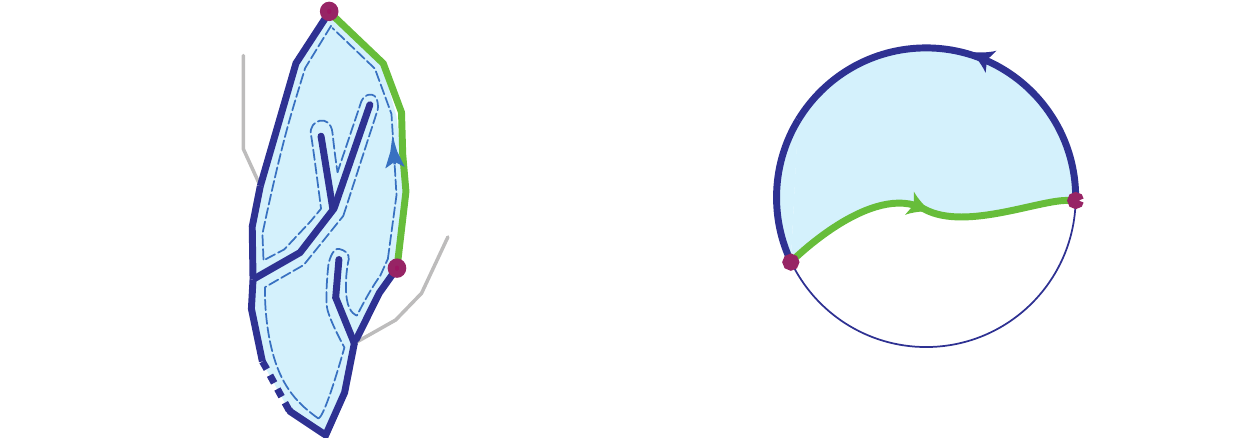} 
   \put(18,1){$\Gamma_i'$}
     \put(22,34){$\ell_o$}
      \put(32.5,15.5){$\ell_i$}
       \put(71.5,23){$D_i$}
       \put(33,25){$\beta_i'$}
        \put(69,14){$\tilde\beta_i'$}
        \put(65,6){$P_T$}
         \put(59.5,12){$\tilde\ell_i$}
          \put(87,18){$\tilde\ell_0$}
           \put(58,23){$\tilde\Gamma_i'$}
   \end{overpic}
   \caption{}\label{fig:pathII}
\end{figure}
$D_k=P_T$. 
The unfolding $P_T\to\ol P_T\subset \R^2$ induces unfoldings $D_i\to\ol D_i\subset \R^2$.
Thus, as was the case for $\partial P_T$  discussed in Section \ref{sec:trees}, there are two congruent ways to map each boundary curve $\partial D_i$ to $\R^2$:  one via the development of $\pi\circ\tilde\Gamma_i'=\Gamma_i'$  and the other via the restriction of the unfolding $\ol D_i$ to $\partial D_i$. So we may record:

\begin{lem}\label{lem:gamma'}
For $1\leq i\leq k$,
the mappings $\partial D_i\to\R^2$ generated by $\overline{\Gamma'}_{i}$ and $\overline{\partial D_i}$ coincide, up to a rigid motion. In particular, 
$\overline{\Gamma'}_{i}$ bounds an immersed disk.\qed
\end{lem}

To prove Proposition \ref{prop:dualbranch}, we need the following lemma whose proof is similar to that of Lemma \ref{lem:mi}. Recall that $j_i$ are local minimizers of $h$ on  $\Gamma$  which traces $T=\pi(\partial P_T)$.  Thus $j_i$ are local minimizers of $h\circ\pi$ on $\partial P_T$. The next observation generalizes this fact. 

\begin{lem}\label{lem:mi2}
Each juncture $j_i$ of $\Gamma$ is a local minimizer of $h\circ\pi$ on $P_T$.
\end{lem}
\begin{proof}
Let $v$, $w$ be vertices of $\Gamma$  which are adjacent to $j_i$. If $j_i=r$, then there is nothing to prove, since $r$ is the absolute minimizer of $h$ on $P$. So  assume that $j_i\neq r$. Then $v\neq w$ by Lemma \ref{lem:mi0}. Consequently  $vw:=[v,j_i, w]$   determines  a pair of sides in $\st_{j_i}$. Let $X\subset \st_{j_i}$ be the set of points whose heights are smaller than $h(j_i)$. Then 
$X$ is connected and is disjoint from $vw$. Thus $X$  lies entirely on one side of $vw$ which will be called the \emph{bottom} side, while the other side will be the \emph{top} side. Recall that $S:=\pi(\widetilde{\st}_{\tilde j_i})$ is one of the sides of $vw$ by Lemma \ref{lem:stv}. We claim  that $S$ is the top side, which is all we need.
To this end  note that the path $j_ir$ of $T$  intersects $X$. So $T$ intersects the interior of the bottom side. But $\inte(S)\cap T=\emptyset$ by \eqref{eq:intS}.
Thus $S$ cannot be the bottom side.
\end{proof}

Now we are ready to prove the main result of this subsection:

\begin{proof}[Proof of Proposition \ref{prop:dualbranch}]
  Let us say a path in $P_T$ is monotone if its projection into $P$ is monotone. We will connect $\tilde\ell_i$ to $\tilde\ell_0$ with a monotone path $\tilde \beta_i'$ in $P_T$ which  intersects $\partial P_T$ only at its end points. Then $\beta_i':=\pi(\tilde \beta_i')$ is the desired path.
We will proceed in two stages:
first (Part I) we construct a monotone path $\tilde \beta_i'$ in $P_T$ which connects $\tilde\ell_i$ to $\tilde\ell_0$, and then (Part II) perturb $\tilde \beta_i'$ to make sure that its interior is disjoint from $\partial P_T$. See Figure \ref{fig:branch} and compare it to Figure \ref{fig:pathII}.

\begin{figure}[h] 
   \centering
   \begin{overpic}[height=1.4in]{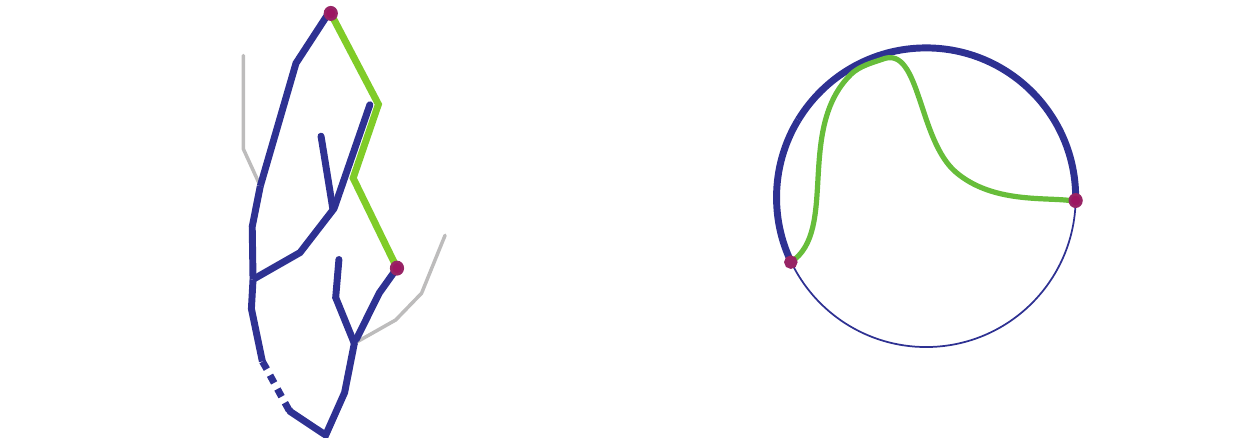} 
     \put(22,35){$\ell_o$}
      \put(31.5,16){$\ell_i$}
        \put(65,5){$P_T$}
         \put(59,13){$\tilde\ell_i$}
          \put(87,18){$\tilde\ell_0$}
   \end{overpic}
   \caption{}
   \label{fig:branch}
\end{figure}

(\emph{Part I}) 
 If $\ell_i=\ell_0$ (i.e., $i=0$), we  set $\tilde\beta_i':=\tilde\ell_0$ and we are done. So suppose  that $\ell_i\neq\ell_0$.
Then there is a vertex $v$  of $P$ adjacent to $\ell_i$ which lies above it. The only edge of $T$ which is adjacent to $\ell_i$ connects to it from below. Thus $\ell_i v$ is not an edge of $T$, and therefore corresponds to a unique edge $\tilde\ell_i\tilde v$ of $P_T$.  This will constitute the first edge of $\tilde\beta_i'$. There are three cases to consider: 


\begin{itemize}
\item[(i)]{$v=\ell_0$},
\item[(ii)]{$v$ is a leaf of $T$ other than $\ell_0$},
\item[(iii)]{$v$ is not a leaf of $T$}.
\end{itemize}
If (i) holds, we are done. If (ii) holds, then we may connect $v$ to an adjacent vertex  $v'$ lying above it to obtain the next edge $\tilde v\tilde v'$ of $\tilde \beta_i'$. If (iii) holds,  then,  by Lemma \ref{lem:mi2}, $v$ cannot be a juncture of $\Gamma$, because it is the highest point of $\ell_i v$. Thus $v$ lies in the interior of a subpath $\ell_{n} j_n$ or $j_n\ell_{n+1}$  of $\Gamma$. In particular, there exists a monotone subpath $v\ell_n$ of $\Gamma^{-1}$ or $v\ell_{n+1}$ of $\Gamma$ which connects $v$ to a leaf $v'$ of $T$ which lies above it. Lifting this path to $\partial P_T$ will extend our path to $\tilde v'$. Now again there are three cases to consider for $v'$, as listed above, and  repeating this process eventually yields the desired path $\tilde\beta_i'$.

(\emph{Part II}) After a subdivision, we may assume that all faces of $P_T$ are triangles. If an edge $E$ of $\tilde\beta_i'$ lies on $\partial P_T$, let $F$ be the face of $\partial P_T$ adjacent to $E$, choose a point $p$ in the interior of $F$ which has the same height as an  interior point of $E$, and replace $E$ with the pair of line segments which connect the vertices of $E$ to $p$, see the left diagram in Figure \ref{fig:branch2}. 
\begin{figure}[h] 
   \centering
   \begin{overpic}[height=0.9in]{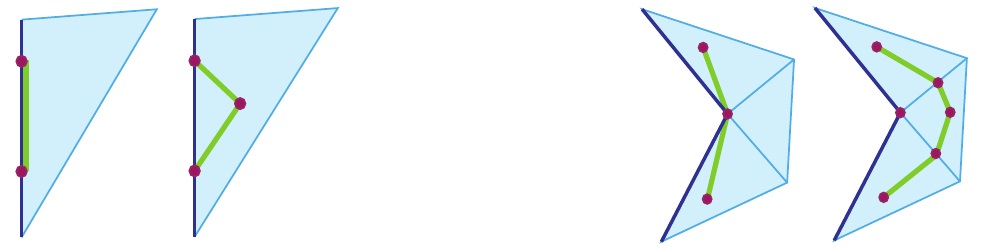} 
   \put(-2,12){$E$}
   \put(25,14){$p$}
   \put(92,13){$p$}
    \put(69,12.5){$v$}
     \put(71.5,19){$a$}
      \put(72,4.75){$b$}
   \end{overpic}
   \caption{}
   \label{fig:branch2}
\end{figure}
Thus we may perturb each edge of $\tilde\beta_i'$ which lies on $\partial P_T$ so that $\tilde\beta_i'$ intersects $\partial P_T$ only at some of its vertices.
Let  $v$ be such a vertex. Further let $a$ (resp. $b$) be a point in the interior of the edge of $\tilde\beta_i'$ adjacent to $v$ which lies above (resp. below) $v$. We need to replace the segment $ab$ of $\tilde\beta_i'$  with another monotone segment in $P_T$ which avoids $v$; see the right diagram in Figure \ref{fig:branch2}. Pick a point $p$ in the interior of the star of $P_T$ at $v$ which has the same height as $v$. It suffices to construct a pair of monotone paths in $\inte(P_T)=P-T$ which connect $a$ and $b$ to $p$. 
The first path may be constructed as follows, and the other path is constructed similarly.
Let $R_a$, $R_b$ be the rays which emanate from $v$ and pass through $a$, $b$ respectively. 
 These rays determine a region $\mathcal{R}$ in the star of $P_T$ at $v$ which is contained  between them. 
  There exists a face $F$ of $P_T$ which contains $a$ and intersects the interior of $\mathcal{R}$.  If $p\in F$, then we connect $a$ to $p$ with a line segment and we are done. If $p\not\in F$, then  $F$ has a unique edge $E$ which lies in the interior of $\mathcal{R}$ and is adjacent to $v$. There is  a point $a'$ in the interior of $E$ which lies below $a$ (because $E$ is adjacent to $v$ which is below $a$). Connect $a$ to $a'$ with a line segment. Next consider the face of $P_T$ which is adjacent to $E$ and is different from $F$. If this face contains $p$ then we  connect $a'$ to $p$ with a line segment and we are done. Otherwise we repeat the above procedure until we reach $p$. 
\end{proof}

\section{Affine Developments of  Monotone Paths}\label{sec:affine}
Here we study the effects of the affine stretchings of $P$ on the developments of its piecewise monotone paths.
The main results of this section are Propositions \ref{prop:path1} and \ref{prop:path2} below. The first proposition shows that affine stretchings of piecewise monotone paths  have piecewise monotone developments, and the second proposition states that  this development is simple if the original curve double covers a monotone path. First we need to prove the following lemmas.
 At each interior vertex  $\gamma_i$ of a path $\Gamma$ in $P$,  let
 $
 \Theta_i
 $
 denote the angle between $\gamma_{i-1}-\gamma_i$, and $\gamma_{i+1}-\gamma_i$ in $\R^3$.
  Further recall that $\theta_i$, ${\theta_i}'$ denote  the left and right angles of $\Gamma$ in $P$.

 \begin{lem}\label{lem:angle0}
 At any interior vertex $\gamma_i$ of a path $\Gamma$ in $P$, we have 
 $\theta_i, \;{\theta_i}'\geq\Theta_i.$
 \end{lem}
 \begin{proof}
 Let $S$ be a unit sphere in $\R^3$ centered at $\gamma_i$, and $\tilde\gamma_{i-1}$, $\tilde\gamma_{i+1}$ be the projections of $\gamma_{i-1}$ and $\gamma_{i+1}$ into $S$ as defined in Section \ref{subsec:sidesAndangles}. Then $\Theta_i$ is the geodesic distance between $\tilde\gamma_{i+1}$ and $\tilde\gamma_{i-1}$ in $S$. So it cannot exceed the length of any curve in $S$ connecting $\tilde\gamma_{i+1}$ and $\tilde\gamma_{i-1}$, including those  which correspond to $\theta_i$, ${\theta_i}'$ .
 \end{proof}

Let $P^\lambda$ denote the 
image of $P$ under the affine stretching 
$
(x,y,z)\mapsto \left(x/\lambda,y/\lambda,z\right)
$.
For any object $X$ associated to $P$ we also let $X^\lambda$ denote the corresponding object of $P^\lambda$. Further, we let $X^\infty$ denote the limit of $X^\lambda$ as $\lambda\to\infty$. In particular note that $P^\infty$ lies on the $z$-axis.
A path  is \emph{monotone} if the heights $h$ of its vertices form a strictly monotone sequence.

\begin{lem}\label{lem:angle1}
Let $\Gamma$ be a monotone path in $P$. Then $\theta_i^\infty= ({\theta_i}')^\infty=\pi$.
\end{lem}
\begin{proof}
For each vertex $\gamma_i$ of $\Gamma$, $h(\gamma_i^\lambda)$ is constant. Thus 
 $h(\gamma_i^\infty)=h(\gamma_i)$. Since $\Gamma$ is monotone, it follows that 
$\gamma_i^\infty$  lies in between $\gamma_{i-1}^\infty$ and $\gamma_{i+1}^\infty$ on the $z$-axis.  So  $\gamma_{i-1}^\infty-\gamma_i^\infty$ and $\gamma_{i+1}^\infty-\gamma_i^\infty$ are antiparallel vectors, 
which yields that 
$
\Theta_i^\infty=\pi.
$
By Lemma \ref{lem:angle0}, $\theta_i^\lambda$, $({\theta_i}')^\lambda\geq\Theta_i^\lambda.$ Thus
$
\theta_i^\infty, ({\theta_i}')^\infty\geq \pi.
$
On the other hand, by \eqref{eq:angles},
$
\theta_i^\infty+ ({\theta_i}')^\infty\leq 2\pi.
$
So  $\theta_i^\infty=({\theta_i}')^\infty=\pi$.
\end{proof}

The last lemma leads to the following observation. 

\begin{lem}\label{lem:angle2}
Let $v$ be a vertex of $P$. Then, $\angle_P(v)^\infty=2\pi$ if $v$ is not  the top or bottom vertex of $P$. Otherwise, $\angle_P(v)^\infty=0$.
\end{lem}
\begin{proof}
The last statement is obvious. To see the first statement note that if $v$ is not an extremum point of $h$, then since $P$ is convex there exists a monotone path $[u,v, w]$ in $P$, where $u$ and $w$ are  adjacent vertices of $v$. Let $\theta$, $\theta'$ be the  angles of this path at $v$. Then $\theta^\infty=(\theta')^\infty=\pi$ by Lemma \ref{lem:angle1}. So $\angle_P(v)^\infty=2\pi$ by \eqref{eq:angles}.
\end{proof}

A path is \emph{piecewise monotone} if it is composed of monotone subpaths, or does not contain any horizontal edges. The last two lemmas yield:

\begin{lem}\label{lem:angle3}
Let $\Gamma$ be a piecewise monotone path in $P$, and $\gamma_i$ be an interior vertex of $\Gamma$. If $\gamma_i$ is a local extremum of $h$ on $\Gamma$, then $\theta_i^\infty$, $({\theta_i}')^\infty=0$ or $2\pi$. Otherwise $\theta_i^\infty=({\theta_i}')^\infty=\pi$.
\end{lem}
\begin{proof}
If $\gamma_i$ is not a local extremum of $h$ (on $\Gamma$), then $[\gamma_{i-1}, \gamma_i,\gamma_{i+1}]$ is a monotone path. Consequently, 
$\theta_i^\infty=({\theta_i}')^\infty=\pi$ by Lemma \ref{lem:angle1} as claimed.
Next suppose  that $\gamma_i$ is a local extremum of $h$. If $\gamma_i$ is the top or bottom vertex of $P$, then $\angle_P(\gamma_i)^\infty=0$, by Lemma \ref{lem:angle2}, which yields that $\theta_i^\infty=({\theta_i}')^\infty=0$ by \eqref{eq:angles}, and again we are done. So suppose that $\gamma_i$ is not an extremum vertex. Then $\angle_P(\gamma_i)^\infty=2\pi$ by Lemma \ref{lem:angle2}, and consequently $\theta_i^\infty+({\theta_i}')^\infty=2\pi$ by \eqref{eq:angles}. So we just need to check that  $\theta_i^\infty=0$ or $2\pi$. 
To see this note that if
$\gamma_i$ is not  simple, then  $\theta_i^\lambda=\angle_P(\gamma_i)^\lambda$, which yields that $\theta_i^\infty=2\pi$,  by Lemma \ref{lem:angle2}. 
So we may assume that $\gamma_i$ is simple.  If $\gamma_i$ is a local maximum (resp. local minimum) of $h$, then there exists a vertex $v$ of $P$ which is adjacent to $\gamma_i$ and lies above (resp. below) it. Consequently, $v$ lies strictly either to the right or left of $[\gamma_{i-1}, \gamma_i,\gamma_{i+1}]$.
Suppose that $v$ lies strictly to left of $[\gamma_{i-1}, \gamma_i,\gamma_{i+1}]$.
Then
$
\theta_i^\lambda=\angle(\gamma_{i-1},\gamma_i,v)^\lambda+\angle(v,\gamma_i,\gamma_{i+1})^\lambda,
$
by Lemma \ref{lem:abc}(ii). But $[\gamma_{i-1},\gamma_i,v]$ and $[v, \gamma_i,\gamma_{i+1}]$ are monotone. Thus by Lemma \ref{lem:angle1},
$
\angle(\gamma_{i-1},\gamma_i,v)^\infty=\pi=\angle(v,\gamma_i,\gamma_{i+1})^\infty.
$
So  $\theta_i^\infty=\pi+\pi=2\pi$.
 If, on the other hand, $v$ lies strictly to the right of $[\gamma_{i-1}, \gamma_i,\gamma_{i+1}]$, then $v$ lies strictly to the left of 
$[\gamma_{i+1}, \gamma_i,\gamma_{i-1}]$, and a similar reasoning shows that $({\theta_i}')^\infty=2\pi$, or $\theta_i^\infty=0$. 
\end{proof}

We will assume that all developments below have initial condition $((0,0), (0,-1))$, as defined in Section \ref{sec:mixed}. A monotone path is \emph{positively} (resp. \emph{negatively}) \emph{monotone} provided that the heights of its consecutive vertices form an increasing (resp. decreasing) sequence.

\begin{prop}\label{prop:path1}
Let $\Gamma$ be a piecewise monotone path in $P$ and $\ol\Gamma$ be a mixed development of $\Gamma$.  
Then  $\ol \Gamma^\infty$ is a path with vertical edges. Furthermore, each subpath of $\ol \Gamma^\infty$ which corresponds to a positively (resp. negatively) monotone subpath of $\Gamma$ will be positively (resp. negatively) monotone.
\end{prop}
\begin{proof}
Recall that $h(\gamma_i^\infty)=h(\gamma_i)$. So since $\Gamma$ is monotone, $\gamma_i^\infty\neq\gamma_{i-1}^\infty$. Then, since  
$\|\ol\gamma_i^\lambda-\ol\gamma_{i-1}^\lambda\|=\|\gamma_i^\lambda-\gamma_{i-1}^\lambda\|$, it follows that
$\|\ol\gamma_i^\infty-\ol\gamma_{i-1}^\infty\|=\|\gamma_i^\infty-\gamma_{i-1}^\infty\|\neq 0.$
So $\ol\gamma_i^\infty\neq\ol\gamma_{i-1}^\infty$, which means that $\ol\Gamma^\infty$ is a path. In particular $\ol\theta_i^\infty$, $({\ol\theta_i}')^\infty$ are well defined, and are limits of $\ol\theta_i^\lambda$, $({\ol\theta_i}')^\lambda$ respectively.
Now Lemma \ref{lem:angle3} quickly completes the argument. 
\end{proof}

The \emph{doubling} of a path $\Gamma=[\gamma_0, \dots, \gamma_k]$ is  the path
$$
D\Gamma:=\Gamma\bullet \Gamma^{-1}=[\gamma_0,\gamma_1,\dots,\gamma_{k-1},\gamma_k, \gamma_{k-1}, \dots,\gamma_1, \gamma_0]=:[\gamma_0,\dots,\gamma_{2k}].
$$
Our next result shows that doublings of monotone paths which end at vertices of $P$ have simple unfoldings once they get stretched enough.

\begin{prop}\label{prop:path2}
Let $\Gamma=[\gamma_0, \dots, \gamma_k]$ be a  monotone path in $P$ such that $\gamma_k$ is a vertex of $P$ different from its top or bottom vertex, and $\overline{D\Gamma}:=(\overline{D\Gamma})_{\gamma^\ell}$ be a mixed development of $D\Gamma$ based at $\gamma_\ell$ for some $0\leq\ell<k$. Then, for sufficiently large $\lambda$:
\begin{enumerate}
\item[(i)]{$\overline{D\Gamma^\lambda}$ is simple.}
\item[(ii)]{The line which passes through  $\ol \gamma_0^\lambda $, $\ol \gamma_{2k}^\lambda$ intersects $\overline{D\Gamma^\lambda}$ at no other point.}
\item[(iii)]{If $\alpha_0^\lambda$, $\beta_0^\lambda$ denote the interior angles of $\overline{D\Gamma^\lambda}\bullet [\ol \gamma_{2k}^\lambda,\ol \gamma_0^\lambda]$ at $\ol \gamma_0^\lambda $, $\ol \gamma_{2k}^\lambda$, then $\alpha^\lambda_0+\beta^\lambda_0<\pi$. Furthermore, $\alpha^\lambda_0$, $\beta_0^\lambda$ may be arbitrarily close to $\pi/2$.}
\end{enumerate}
\end{prop}
\begin{proof}
We proceed by induction on the number of edges of $\Gamma$. Clearly the proposition holds when $\Gamma$ has only one edge. Suppose that it holds for  the subpath $\Gamma_1^\lambda:=[\gamma_1^\lambda,\dots, \gamma_{k}^\lambda]$ of $\Gamma^\lambda$. Then we claim that it also holds for $\Gamma^\lambda$. Henceforth we will assume that $\lambda$ is arbitrarily large and  drop the explicit reference to it. 
Let $L_1$ be the line passing through the end points $\ol \gamma_1$, $\ol \gamma_{2k-1}$ of $\overline{D\Gamma_1}$, and $o$ be the midpoint of $\ol \gamma_1\ol \gamma_{2k-1}$. We may assume, after rigid motions, that  $o$ is fixed, $L_1$ is  horizontal, and $\overline{D\Gamma_1}$ lies above $L_1$, see Figure \ref{fig:trapezoid}.  Furthermore, since by assumption $\gamma_k$ is not the top or bottom vertex of $P$, we may assume that the left angle of $D\Gamma$ at $\gamma_k$ (which coincides with the total angle of $P$ at $\gamma_k$) is arbitrarily close to $2\pi$ by Lemma \ref{lem:angle2}.
 \begin{figure}[h] 
   \centering
   \begin{overpic}[height=1in]{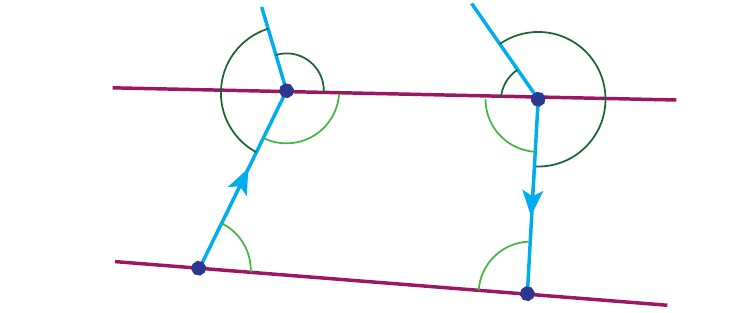} 
   \put(8,5){$L_0$}
   \put(8,28){$L_1$}
     \put(25,1){$\ol\gamma_0$}
     \put(68,-2.5){$\ol\gamma_{2k}$}
      \put(31,32){$\ol\gamma_1$}
       \put(70,32){$\ol\gamma_{2k-1}$}
        \put(34,9){$\alpha_0$}
          \put(59,7){$\beta_0$}
          \put(40,20){$\alpha_1'$}
          \put(61,20){$\beta_1'$}
          \put(40,35){$\alpha_1$}
          \put(61,32){$\beta_1$}
          \put(25,23){$\ol\theta_1$}
          \put(81,22){$\ol\theta_{2k-1}$}
   \end{overpic}
   \caption{}
   \label{fig:trapezoid}
\end{figure}
Then it follows that $\ol\gamma_{2k-1}$ lies to the right of $\ol\gamma_1$ on $L_1$, just as depicted in Figure \ref{fig:trapezoid}. Now we claim that $\ol \gamma_0$, $\ol \gamma_{2k}$ lie below $L_1$. To see this, let $\alpha_1$, $\beta_1$ be the interior  angles of  $\overline{D\Gamma_1}\bullet [\ol \gamma_{2k-1},\ol \gamma_1]$ at $\ol \gamma_1$, $\ol \gamma_{2k-1}$ respectively. 
Further let $\ol\theta_1$, $\ol\theta_{2k-1}$ denote respectively the left angles of $\overline{D\Gamma}$ at $\ol \gamma_1$ and $\ol \gamma_{2k-1}$. We may assume that $\alpha_1$, $\beta_1\approx\pi/2$ by the inductive hypothesis on $\Gamma_1$. 
By Lemma \ref{lem:angle1}, we may also assume that 
$
\ol\theta_1,\ol\theta_{2k-1}\approx \pi.
$
 So
\begin{equation}\label{eq:alpha1beta1}
\alpha_1+\ol\theta_1\approx\frac{3\pi}{2},  \quad\quad      \text{and}    \quad\quad    \beta_1+\ol\theta_{2k-1}\approx\frac{3\pi}{2},
\end{equation}
which show that $\ol \gamma_0$, $\ol \gamma_{2k}$  lie below $L_1$ as claimed. 
Next we show that  $\ol \gamma_1 \ol \gamma_0$, $\ol \gamma_{2k-1}\ol \gamma_{2k}$ do not intersect, which will establish  (i). To this end it suffices to check that $\alpha_1'+\beta_1'\geq\pi$, where 
$$
\alpha_1':=2\pi-\alpha_1-\ol\theta_1, \quad \text{and} \quad \beta_1':=2\pi-\beta_1-\ol\theta_{2k-1}.
$$
There are two cases to consider:  either
 $\ol\theta_1=\theta_1$ or ${\ol\theta_1}'={\theta_1}'$ by the definition of mixed development. If  $\ol\theta_1=\theta_1$, then
$$
\ol\theta_1+\ol\theta_{2k-1}=\theta_1+\theta_{2k-1}=\theta_1+{\theta_1}'=\angle_P(\gamma_1)\leq 2\pi,
$$ 
where the identity $\theta_{2k-1}={\theta_1}'$ used here follows from the definition of $D\Gamma$. If, on the other hand, ${\ol\theta_1}'={\theta_1}'$, then 
$$
\ol\theta_1+\ol\theta_{2k-1}=2\pi-{\ol\theta_1}'+\ol\theta_{2k-1}=
2\pi-{\theta_1}'+\theta_{2k-1}= 2\pi.
$$
So we always have $\ol\theta_1+\ol\theta_{2k-1}\leq2\pi$.
Also note that $\alpha_1+\beta_1<\pi$ by the inductive hypothesis on $\Gamma_1$.
Thus it follows that, 
\begin{equation}\label{eq:alphabeta}
\alpha_1'+\beta_1'=4\pi-(\alpha_1+\beta_1)-(\ol\theta_1+\ol\theta_{2k-1})>4\pi-\pi-2\pi=\pi,
\end{equation}
 as desired.
 To establish  (ii), let $L_0$ be the line passing through $\ol \gamma_0$, $\ol \gamma_{2k}$. By \eqref{eq:alphabeta}  the quadrilateral $Q:=\ol \gamma_0 \ol \gamma_1 \ol \gamma_{2k-1} \ol \gamma_{2k}$ is convex. Thus $\ol \gamma_1$, $\ol \gamma_{2k-1}$ lie on the same side or ``above" $L_0$. It remains to check  that $\overline{D\Gamma_1}$ is disjoint from $L_0$. To this end note that the length of $\overline{D\Gamma_1}$ is bounded from above, since affine stretchings do not increase lengths. So $\overline{D\Gamma_1}$ is contained in a half disk $H$ of some constant radius which lies above $L_1$ and   is centered at $o$. Further  $\ol\gamma_1\ol\gamma_0$ and $\ol\gamma_{2k-1}\ol\gamma_{2k}$ are almost orthogonal to $L_0$ by \eqref{eq:alpha1beta1}, and they have the same length, which is bounded from below (by $|h(\gamma_1)-h(\gamma_0)|$). Thus $L_0$ is nearly parallel to $L_1$ while its distance from $o$ is  bounded from below. So $L_0$ will be disjoint from $H$.
Finally, (iii) follows immediately from \eqref{eq:alphabeta}, since $Q$ is a simple quadrilateral and thus the sum of its interior angles  is $2\pi$. 
\end{proof}

 \section{Proof of Theorem \ref{thm:main}}\label{sec:proof}
 For convenience, we may assume that $u=(0,0,1)$.
Let $\Gamma:=\Gamma_T$ be the  path which traces $T$ as defined in Section \ref{sec:trees}. Recall that, as we showed in Section \ref{sec:monotonetrees}, $\Gamma$ admits a decomposition into monotone subpaths:
 \begin{equation*}
 \Gamma=\ell_0j_0\bullet j_0\ell_1\bullet\dots\bullet\ell_{k-1}j_{k-1}\bullet j_{k-1}\ell_0.
 \end{equation*}
 Also recall that $\ell_ij_i$ are negatively monotone, and $j_i\ell_{i+1}$ are  positively monotone.  
By Proposition \ref{prop:disk} we just need to show that the development $\ol\Gamma^\lambda$ is simple for large $\lambda$. To this end, we first record how large $\lambda$ needs to be, and then proceed  by induction on the number of  leaves of $T$. 
 
 \subsection{Fixing the stretching factor $\lambda$}\label{subsec:proof1}
 Let $\Gamma_i$, $\Gamma'_i$ be the paths  defined in Section \ref{sec:monotonetrees}, and recall that these paths also admit decompositions into monotone subpaths:
 \begin{eqnarray*}
  \Gamma_i&=&\ell_0j_0\bullet j_0\ell_1\bullet\dots\bullet\ell_{i-1}j_{i-1}\bullet j_{i-1}\ell_i\bullet\ell_i r,\quad\; 0\leq i\leq k-1, \\ 
     \Gamma_i'&=&\ell_0j_0\bullet j_0\ell_1\bullet\dots\bullet\ell_{i-1}j_{i-1}\bullet j_{i-1}\ell_i\bullet\ell_i \ell_0, \quad 1\leq i\leq k. 
 \end{eqnarray*}
Let $\Gamma_i^\lambda$, $(\Gamma_i')^\lambda$ 
 denote the affine stretching of these paths, and $\ol\Gamma_i^\lambda$, $(\overline{\Gamma_i'})^\lambda$ be their corresponding developments with initial condition $((0,0),(0,-1))$, as in Section \ref{sec:affine}.     We need to choose $\lambda$ so large that:
 
\smallskip
 \begin{enumerate}
 \item[\textbf{(C1)}]{\emph{For each positively (resp. negatively) monotone subpath of $\Gamma_i$ or $\Gamma_i'$ the corresponding subpath of $\ol\Gamma_i^\lambda$ or $(\overline{\Gamma_i'})^\lambda$ is  positively (resp. negatively) monotone}}.\\
 \item[\textbf{(C2)}]{\emph{$(\ol\Gamma_i^\lambda)^{-1}\circ\ol\Gamma_{i+1}^\lambda$ is simple and lies on one side of the line $L^\lambda$ passing through its end points. Furthermore, $L^\lambda$ is not vertical (see Figure \ref{fig:lambda})}}.
 \end{enumerate}
   \begin{figure}[h] 
   \centering
   \begin{overpic}[height=1.5in]{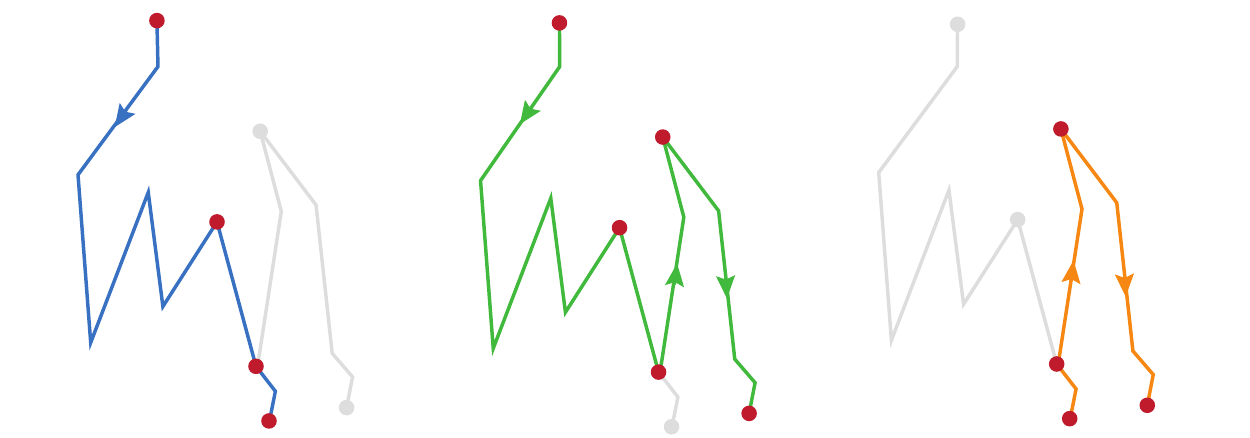} 
   \put(9,33){$\ol\ell_0$}
   \put(41,33){$\ol\ell_0$}
    \put(16,19){$\ol\ell_i$}
   \put(48.5,18.5){$\ol\ell_i$}
   \put(17,5){$\ol j_i$}
   \put(49,5){$\ol j_i$}
     \put(81,5){$\ol j_i$}
      \put(18,0){$\ol r_i$}
     \put(82.5,0){$\ol r_i$}
      \put(61,1){$\ol r_{i+1}$}
      \put(93,1){$\ol r_{i+1}$}
       \put(50,26){$\ol \ell_{i+1}$}
       \put(82,26.5){$\ol \ell_{i+1}$}
        \put(4,24){$\ol\Gamma_i$}
   \put(33,24){$\ol\Gamma_{i+1}$}
    \put(65,24){$\ol\Gamma_i^{-1}\circ\ol\Gamma_{i+1}$}
   \end{overpic}
   \caption{}
   \label{fig:lambda} 
\end{figure}

\noindent
To see that (C1) holds let $\gamma_j$, $\ol\gamma_j^\infty$ denote the vertices of $\Gamma$, $\ol\Gamma^\infty$, and set
 $$
 0<\epsilon<\frac{1}{2}\inf_j\|\ol \gamma_j^\infty-\ol \gamma_{j-1}^\infty\|=\frac{1}{2}\inf_j| h(\gamma_j)- h(\gamma_{j-1})|.
 $$
 Choose $\lambda$ so large that $\|\ol\gamma_j^\lambda-\ol\gamma_j^\infty\|\leq\epsilon$.
Then
$\ol\gamma_j^\lambda$ lies below (resp. above) $\ol\gamma_{j-1}^\lambda$
 if and only if $\ol\gamma_j^\infty$ lies below (resp. above) $\ol\gamma_{j-1}^\infty$. Thus monotone subpaths of $\ol\Gamma_i^\lambda$ correspond to those of $\ol\Gamma_i^\infty$, which by Proposition \ref{prop:path1} correspond to the monotone subpaths of $\Gamma_i$. Similarly we may obtain an estimate for $\lambda$ in $(\overline{\Gamma_i'})^\lambda$.
To see that (C2) holds note that, by Proposition \ref{cor:branches}, 
$$
(\ol\Gamma_i^\lambda)^{-1}\circ\ol\Gamma_{i+1}^\lambda\equiv\overline{(\beta_{i+1}^\lambda)^{-1}\bullet \beta_{i+1}^\lambda})_{j_i^\lambda}\equiv(\overline{D\beta_{i+1}^\lambda})_{j_i^\lambda}.
$$
 So, since $\beta_i$ are monotone, it follows from Proposition \ref{prop:path2} that the right hand side of the above expression is simple and lies on one side of the line $L^\lambda$ passing through its end points. Further, $L^\lambda$ becomes arbitrarily close to meeting  
$(\ol\Gamma_i^\lambda)^{-1}\circ\ol\Gamma_{i+1}^\lambda$ orthogonally, as $\lambda$ grows large. At the same time, the edges of  $(\ol\Gamma_i^\lambda)^{-1}\circ\ol\Gamma_{i+1}^\lambda$ become arbitrarily close to being vertical, by Proposition \ref{prop:path1}. Thus  $L^\lambda$ cannot be vertical for  large $\lambda$. 
For the rest of the proof we fix $\lambda$ to be so large that  (C1), (C2) hold, and  drop the explicit reference to it.

\subsection{The inductive step}\label{subset:proof2}
It remains to show that $\ol\Gamma$ is simple. To this end  recall the definition of weakly monotone from Section \ref{sec:immersion}, and observe  that:
 
 \begin{lem}\label{lem:Di}
For $0\leq i\leq k-1$, if $\ol \Gamma_i$ is weakly monotone,  then  $\overline{\Gamma'}_{i+1}$ is simple.  
\end{lem}

\begin{proof}
By Lemma \ref{lem:gamma'}, $\overline{\Gamma'}_{i+1}$ bounds an immersed disk.  So, by Proposition \ref{cor:disk}, it suffices to show that $\overline{ \Gamma'}_{i+1}$ admits a decomposition into a pair of weakly monotone curves. Indeed,
$
\overline{ \Gamma'}_{i+1}=\overline{\ell_0  j_i}\bullet\overline{j_i\ell_{0}},
$
see Figure \ref{fig:path2}. 
Note that $\overline{\ell_0  j_i}$ is weakly monotone, because it is a subpath of $\ol \Gamma_{i}$. 
To show that $\overline{j_i\ell_{0}}$ is  also weakly monotone,  via (C1), it suffices to check that $j_i\ell_0$ is  monotone. This is  the case, since 
$
j_i\ell_0=j_i\ell_{i+1}\bullet \ell_{i+1}\ell_0=j_i\ell_{i+1}\bullet \beta_{i+1}',
$
and $j_i\ell_{i+1}$, $\beta_{i+1}'$ are both positively monotone.
 \end{proof}
 
 \begin{figure}[h] 
   \centering
   \begin{overpic}[height=1.5in]{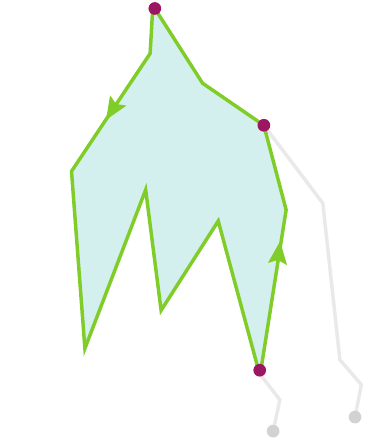} 
   \put(25,98){$\ol\ell_0$}
   \put(56,77){$\ol\ell_{i+1}$}
   \put(33,64){$\ol D_{i+1}$}
    \put(-3,54){$\ol \Gamma_{i+1}'$}
    \put(50,14){$\ol j_i$}
   \end{overpic}
   \caption{}
   \label{fig:path2}
\end{figure}

Now recall that $\Gamma_k'=\Gamma$ by \eqref{eq:defGamPrime}. Thus, by Lemma \ref{lem:Di}, to complete the proof of Theorem \ref{thm:main} it suffices to show that $\ol \Gamma_{k-1}$ is weakly monotone.
By (C1),
$\ol \Gamma_0$ is  monotone, since $\Gamma_0=\beta_0$ is monotone. So  it  remains to show:

\begin{lem}\label{prop:induction}
For $0\leq i\leq k-2$, if $\ol \Gamma_i$ is weakly monotone, then so is  $\ol \Gamma_{i+1}$.
\end{lem}

\noindent
To establish this lemma, let $a$ be  a point on the $y$-axis which lies above all paths $\ol\Gamma_i$, $\overline{\Gamma_i'}$. Further let $\ol r_i$ be the final point of $\ol\Gamma_i$ and $b_i$ be  a point
with the same $x$-coordinate as $\ol r_i$ which lies below all paths $\ol\Gamma_j$, $\overline{\Gamma_j'}$. We may also assume that all $b_i$ have the same height. Now    set
$$
\widehat \Gamma_i:=a\ol\ell_0\bullet\ol \Gamma_i\bullet \ol r_ib_i.
$$ 
 Then $\ol \Gamma_i$ is weakly monotone if and only if $\widehat \Gamma_i$ is simple. Thus to prove Lemma \ref{prop:induction}, we need to show that $\widehat \Gamma_{i+1}$ is simple,  if $\widehat \Gamma_i$ is simple.  To this end note that 
 $\widehat\Gamma_{i+1}=a \ol\ell_{i+1}\bullet\ol\ell_{i+1}b_{i+1}$. Thus it suffices to check  that 
\begin{enumerate}
\item[(I)] $a \ol\ell_{i+1}$ and $\ol\ell_{i+1}b_{i+1}$  are each simple,
\item[(II)] $a \ol\ell_{i+1}\cap\ol\ell_{i+1}b_{i+1}=\{\ol\ell_{i+1}\}$.
\end{enumerate}

\subsection{Proof of the inductive step}
It remains to establish  items (I) and (II) above subject to the assumption  that $\widehat \Gamma_i$ is simple, or $\ol\Gamma_i$ is weakly monotone,  in which case $\overline{\Gamma'}_{i+1}$ is also simple by Lemma \ref{lem:Di}.

\subsubsection{(I)}
First we check that $\ol\ell_{i+1}b_{i+1}$ is simple. Note that
$$
\ol\ell_{i+1}b_{i+1}=\ol\ell_{i+1}\ol r_{i+1}\bullet\ol r_{i+1}b_{i+1}=\overline{\ell_{i+1} r}\bullet \ol r_{i+1}b_{i+1},
$$
see the right diagram in Figure \ref{fig:path}.
Recall that  $\ol r_{i+1}b_{i+1}$ is negatively monotone by the definition of $b_{i+1}$. Further, by (C1), $\overline{\ell_{i+1} r}$ is negatively monotone as well, since $\ell_{i+1} r$ is negatively monotone by the definition of $\Gamma_{i+1}$. So $\ol\ell_{i+1}b_{i+1}$ is monotone and therefore simple.
\begin{figure}[h] 
   \centering
   \begin{overpic}[height=1.75in]{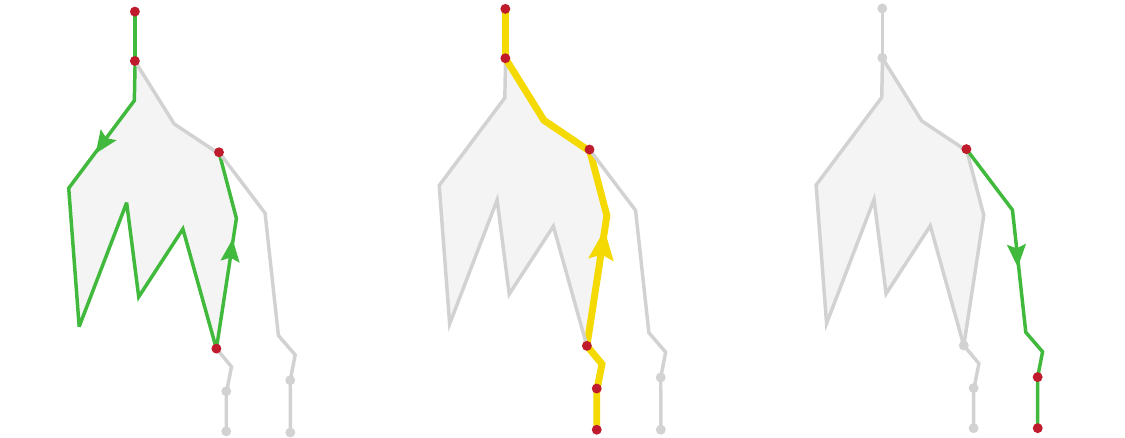} 
    \put(8,33){$\ol\ell_0$}
   \put(41,33){$\ol\ell_0$}
    \put(9.5,38){$a$}
   \put(42.5,38){$a$}
   \put(15,7.5){$\ol j_i$}
   \put(48,8){$\ol j_i$}
      \put(93.5,5){$\ol r_{i+1}$}
       \put(93.5,0){$b_{i+1}$}
       \put(50,4){$\ol r_{i}$}
       \put(50,-0.5){$b_{i}$}
       \put(17,27){$\ol \ell_{i+1}$}
       \put(50.5,19){$A$}
        \put(51,27){$\ol \ell_{i+1}$}
       \put(84,27){$\ol \ell_{i+1}$}
   \end{overpic}
   \caption{}
   \label{fig:path}
\end{figure}
Next we establish  the simplicity of $a\ol\ell_{i+1}$. Note that 
$
a \ol\ell_{i+1}
=a \ol j_i\bullet\ol j_i\ol \ell_{i+1},
$ 
and $a \ol j_i$ is simple because it is a subpath of $\widehat\Gamma_i$, see the left diagram in Figure \ref{fig:path}. Furthermore, $\ol j_i\ol \ell_{i+1}$ is simple as well, because it is a subpath of $\overline{\Gamma'}_{i+1}$. It remains to check that 
 $
a \ol j_i\cap \ol j_i\ol \ell_{i+1}=\{\ol j_i\}.
 $ 
  To see this note that $a \ol j_i=a\ol\ell_0\bullet\ol\ell_0 \ol j_i$. Thus it suffices to show that
  $$
 \ol\ell_{0}\ol j_i\cap \ol j_i\ol\ell_{i+1}=\{\ol j_i\},\quad\text{and}\quad a\ol\ell_0\cap \ol j_i\ol\ell_{i+1}=\emptyset.
  $$
  The first equality  holds because $\ol\ell_0 \ol j_i$ and $\ol j_i\ol\ell_{i+1}$ are both subpaths of $\overline{\Gamma'}_{i+1}$. To see the second equality   note that $\ol j_i\ol\ell_{i+1}=\overline{j_i\ell_{i+1}}$ is positively monotone by (C1) while $a\ol\ell_0$ is negatively monotone by definition of $a$. So it suffices to check that $\ol\ell_{i+1}$ lies below $\ol\ell_0$. This is the case because 
 $
\ol\ell_{i+1}\ol\ell_{0}=\overline{\ell_{i+1}\ell_0}=\overline{\beta_{i+1}'},
$
and $\beta_{i+1}'$ is positively monotone. Thus $\ol\ell_{i+1}\ol\ell_0$ is positively monotone by (C1).

\subsubsection{(II)} Let
$$
A:=b_i\ol r_i\bullet (\overline{r_ij_i })_{\ol\Gamma_i^{-1}} \bullet (\overline{j_i\ell_0})_{\overline{\Gamma'}_{i+1}}\bullet \ol\ell_0a,
$$
 see the middle diagram in Figure \ref{fig:path}. 
Since each of the paths in this composition is positively monotone,  $A$ is   simple. Let $S\subset\R^2$ be the slab contained between the horizontal line passing through $a$ and the horizontal line on which all $b_i$ lie. Then $S-A$ will have precisely two components, whose closures will be called the \emph{sides} of $A$, and may be distinguished as the left  and the right side in the obvious way. To establish claim (II) above it suffices to show:
 \begin{itemize}
  \item[(i)]{$a\ol\ell_{i+1}$ lies to the left of $A$}, 
  \item[(ii)]{one point of $\ol\ell_{i+1}b_{i+1}$ lies strictly to the right of $A$},
  \item[(iii)]{$\ol\ell_{i+1}b_{i+1}\cap A=\{\ol\ell_{i+1}\}$}. 
 \end{itemize}
 Indeed, (ii) and (iii) show that all of $\ol\ell_{i+1}b_{i+1}$  lies to the right of $A$, because $\ol\ell_{i+1}b_{i+1}$ is connected and lies in the slab $S$. This together with (i) show that $a\ol\ell_{i+1}$ and $\ol\ell_{i+1}b_{i+1}$ may intersect only along $A$, and then (iii) ensures that the intersection is $\ol\ell_{i+1}$. It remains to establish each of the three items listed above:

 \emph{(i)} We have $a\ol\ell_{i+1}=a\ol\ell_0\bullet \ol\ell_0\ol j_i\bullet\ol j_i\ol\ell_{i+1}$. Note that $a\ol\ell_0$ and $\ol j_i\ol\ell_{i+1}$ lie on $A$. Thus it remains to check that  $\ol\ell_0\ol j_i$ lies to the left of $A$. We have
 $$
 A=b_i\ol j_i\bullet \ol j_i\ol\ell_0\bullet \ol\ell_0a.
 $$
Note that 
 $\ol\ell_0\ol j_i$ meets $\ol\ell_0a$ and $b_i\ol j_i$ only at its end points, since all these paths lie on $\widehat\Gamma_i$. Further $\ol\ell_0\ol j_i$ meets  $\ol j_i\ol \ell_0$ again only at its end points, since these paths lie on $\ol{\Gamma_{i+1}'}$. So $A$ meets $\ol\ell_0\ol j_i$  only at its end points. It suffices to show then that a point in the interior of $\ol\ell_0\ol j_i$ lies on the left of $A$. This is so, because  
 $\ol D_{i+1}$ lies on the left of $\overline{\Gamma'}_{i+1}$ and the orientations of $A$ and $\overline{\Gamma'}_{i+1}$ agree where they meet.

\emph{(ii)} Near $\ol\ell_{i+1}$, $A$ coincides with  $\overline{\Gamma'}_{i+1}$. Let $C$ be a circle centered at $\ol\ell_{i+1}$ whose radius is so small that it intersects $\overline{\Gamma'}_{i+1}$ and $A$ only at two points, see the right diagram in Figure \ref{fig:ii}. 
\begin{figure}[h] 
   \centering
   \begin{overpic}[height=1in]{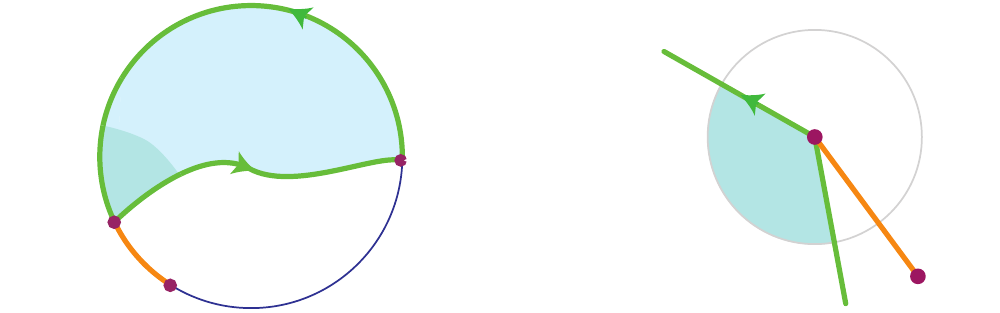} 
   \put(2,24){$\tilde\Gamma_{i+1}'$}
   \put(21,21){$D_{i+1}$}
    \put(11.5,12.5){$U$}
    \put(3,7){$\tilde\ell_{i+1}$}
     \put(93,20){$C$}
     \put(10,2){$E$}
     \put(41,14){$\tilde\ell_{0}$}
      \put(75,12.5){$\ol U$}
       \put(79,20){$\ol\ell_{i+1}$}
        \put(90,7){$\ol E$}
         \put(59,24){$\ol\Gamma_{i+1}'$}
   \end{overpic}
   \caption{}
   \label{fig:ii}
\end{figure}
Then  there exists a neighborhood $U$ of $\tilde\ell_{i+1}$ in $D_{i+1}$ whose image $\ol U$ coincides with the  left side of $\overline{\Gamma'}_{i+1}$ in $C$. Consider the edge $\ol E$ of $\ol\ell_{i+1}b_{i+1}$ which is adjacent to $\ol\ell_{i+1}$, and let $E$ be the corresponding edge of $\tilde\Gamma_{i+1}$ in $\partial P_T$, see the left diagram in Figure \ref{fig:ii}. We claim that there is a point of $\ol E$ inside $C$ which lies strictly to the right of $\overline{\Gamma'}_{i+1}$, or is disjoint from $\ol U$, which is all we need. To see this recall that  the unfolding $P_T\to \R^2$ is locally one-to-one. Thus it suffices to note that the interior of $E$ is disjoint from $D_{i+1}$.  Indeed, since $E\subset\partial P_T$,
$
E\cap D_{i+1}=E\cap D_{i+1}\cap\partial P_T=E\cap\tilde\ell_0\tilde\ell_{i+1}=\{\tilde\ell_{i+1}\}.
$

\emph{(iii)} Since $\ol\ell_{i+1}b_{i+1}$ is negatively monotone,  it may intersect $A$ only along its subpath which lies below $\ol\ell_{i+1}$, that is  $b_i\ol\ell_{i+1}$. So it suffices to check that $b_i\ol\ell_{i+1}\cap\ol\ell_{i+1}b_{i+1}=\{\ell_{i+1}\}$, or
$b_ib_{i+1}:=b_i\ol\ell_{i+1}\bullet\ol\ell_{i+1}b_{i+1}$ is simple. To see this note that
$$
b_ib_{i+1}=b_i\ol r_i\bullet \overline{r_i r_{i+1}}\bullet \ol r_{i+1}b_{i+1},
$$
see Figure \ref{fig:rr}.
The first and  third paths in this decomposition are  simple. Further, 
$
\overline{r_ir_{i+1}}=\ol\Gamma_i^{-1}\circ\ol\Gamma_{i+1}
$
 which is also simple by (C2). Furthermore, again by (C2), $\overline{r_i r_{i+1}}$ lies above the line $L$ passing through its end points, while $b_i\ol r_i$ and $\ol r_{i+1}b_{i+1}$ lie below $L$ (``above" and ``below" here are all well-defined, since $L$ is not vertical by (C2)). So $b_ib_{i+1}$ is simple, as claimed.

\begin{figure}[h] 
   \centering
   \begin{overpic}[height=1.1in]{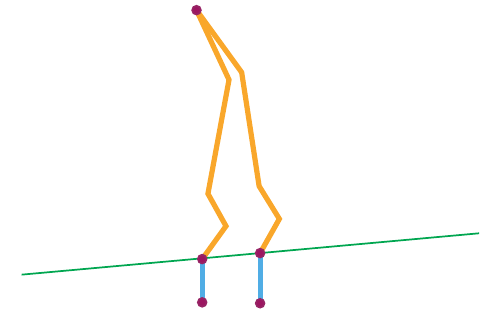} 
   \put(24,60){$\ol\ell_{i+1}$}
    \put(33,13){$\ol r_{i}$}
     \put(33,0){$b_{i}$}
     \put(56,12){$\ol r_{i+1}$}
     \put(54,0){$b_{i+1}$}
       \put(-2,6){$L$}
   \end{overpic}
   \caption{}
   \label{fig:rr}
\end{figure}

\section*{Appendix A: More on Embeddedness of Immersed Disks}

Here we generalize Proposition \ref{cor:disk}, in case it might  be useful in making further progress on \Durers problem.
We say $R\subset\R^2$ is a \emph{ray} emanating from  $p$ if there exists a continuous one-to-one map $r\colon[0,\infty)\to\R^2$ such that $r([0,\infty))=R$, $r(0)=p$ and $\|r(t)-p\|\to\infty$  as $t\to\infty$. Also, as before, for any $X\subset D$, and mapping $f\colon D\to\R^2$, we set $\ol X:=f(X)$, and say $\ol X$ is simple if $f$ is one-to-one on $X$.

\begin{thm}\label{thm:disk}
Let $D\overset{f}{\to}\R^2$ be an  immersion. Suppose there are $k\geq 2$ distinct points $p_i$, $i\in \Z_k$, cyclically arranged in $\partial D$ such that $\overline{p_ip_{i+1}}$ is simple. Further suppose that there are rays $R_i\subset\R^2$ emanating from  $\ol p_i$  such that  
\begin{enumerate}
\item[(i)]{$R_i\cap R_{i+1}=\emptyset$,}
\item[(ii)]{$R_i\cap\overline{p_{i-1}p_{i+1}}=\{\ol p_i\}$,}
\item[(iii)] there is an open neighborhood $U_i$ of  $p_i$ in $D$ and a point $r_i\in R_i-\{\ol p_i\}$ such that $\ol U_i\cap \ol p_ir_i=\{\ol p_i\}$. 
\end{enumerate}
Then $\ol D$ is simple.
\end{thm} 

To prove this theorem we need  a pair of lemmas, which follow from  the theorem on the invariance of domain (if $M$ and $N$ are  manifolds of the same dimension and without boundary, $U\subset M$ is open, and $f\colon U\to N$ is a one-to-one continuous map, then $f(U)$ is open in $N$). The first lemma also uses the fact that  a simply connected manifold admits only trivial coverings.

\begin{lem}\label{lem:preX}
Let $M$ be a compact connected surface, and $M\overset{f}{\to}\R^2$ be an immersion. Suppose that $\overline{\partial M}$ lies on a simple closed curve $C\subset\R^2$. Then $\ol M$ is simple.
\end{lem}
\begin{proof}
There is a homeomorphism $\phi\colon\R^2\to\R^2$ which maps $C$ to $\S^1$, by the theorem of Schoenflies \cite{moise}. So we may assume that $\overline{\partial M}\subset\S^1$, after replacing $f$ with $\phi\circ f$.  Since $M$ is compact, it contains a point $x$ which maximizes  $\|f\|\colon M\to\R$. 
By invariance of domain, $\overline{\inte(M)}$ is open in $\R^2$. Thus it follows that $x\in\partial M$, or $\ol x\in\S^1$, which in turn  implies  that $\|f\|\leq 1$, or $\ol M\subset D$.
Now since $\overline{\partial M}\subset \S^1=\partial D$,  $f\colon M\to D$ is a local homeomorphism. To see this let $U$ be an open neighborhood  in $M$ such that $f$ is one-to-one on the closure $\cl U$ of $U$. Then $\cl U$ is homeomorphic to $\overline{\cl U}$ (any one-to-one continuous map from a compact space into a Hausdorff space is a homeomorphism onto its image). So $U$ is homeomorphic to $\ol U$. Further since
$\overline{(U\cap\partial M)}\subset \partial D$ it follows that $\ol U$ is open in $D$, as claimed.
Now since $M$ is compact and $D$ is connected,  $f$ is a covering map (this is a basic topological fact, e.g., see \cite[p. 375]{docarmo:curves}). But $D$ is simply connected, and $M$ is connected; therefore, $f$ is one-to-one.
\end{proof}

 For every $x\in\R^2$ let $B_r(x)$ denote the (closed) disk of radius $r$ centered at $x$. Then for any  $X\subset\R^2$, we set $B_r(X):=\cup_{x\in X}B_r(x)$.

\begin{lem}\label{lem:X}
Let $D\to\R^2$ be an immersion, and $A\subset\partial D$ be a closed set such that $\ol A$ is simple. Then for every closed connected set $X\subset \inte(A)$ and $\epsilon>0$, there exists a connected  open neighborhood $U$ of $X$ in $D$ such that $\ol U$ is simple and lies in $B_\epsilon(\ol X)$. Furthermore, 
$\overline{U-A}$ is  open, connected, and  $\overline{U-A}\cap\overline A=\emptyset$.
\end{lem}

\begin{proof}
Let $U:=\inte (B_\delta(X))\cap D$, see Figure \ref{fig:D&X}. 
We claim that if $\delta>0$ is sufficiently small, then $U$ is the desired set.
Indeed (for small $\delta$)  $\ol U\subset B_\epsilon(\ol X)$, since $D\to\R^2$ is continuous and $X$ is compact.
Further, since $D\to\R^2$ is locally one-to-one, $X$ is compact, and $\ol X$ is simple, it follows that $\ol U$ is simple (this is a  basic fact, e.g., see \cite[p. 345]{spivak:v1}). 
 \begin{figure}[h] 
   \centering
   \begin{overpic}[height=0.8in]{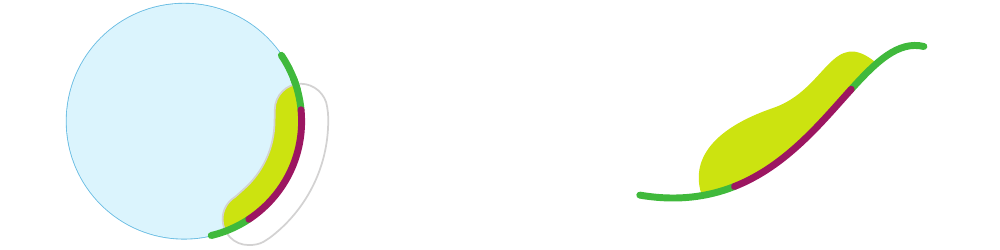} 
   \put(17,12){$D$}
   \put(24,6){$U$}
    \put(29,5){$X$}
     \put(30,19){$A$}
      \put(73,11){$\ol U$}
       \put(81,6){$\ol X$}
        \put(94,19){$\ol A$}
   \end{overpic}
   \caption{}
   \label{fig:D&X}
\end{figure}
Next note that since $X\subset\inte(A)$,  $X$ is disjoint from $\partial D-\inte(A)$, which is compact. Thus $U$ will be disjoint from $\partial D-A$ as well. Consequently 
$U-A=U-\partial D$
which  is open in $\R^2$. So, since $\overline{U-A}$ is simple, it follows from the  invariance of domain that $\overline{U-A}$ is open, and it is connected as well since $U -A$ is connected.
 Finally note that if we set $V:=\inte (B_\delta(A))\cap D$, then $\ol V$ will be simple, just as we had argued earlier for $\ol U$. So, since $U$, $A\subset V$, we have
$
\overline{(U-A)}\cap\overline{A}=\overline{(U-A)\cap A}=\emptyset.
$
\end{proof}

Now we are ready to prove the main result of this section:

\begin{proof}[Proof of Theorem \ref{thm:disk}]
We will extend $f$ to an immersion $\tilde f\colon M\to\R^2$ where $M$ is a compact connected surface containing $D$, $\tilde f=f$ on $D$, and $\tilde f(\partial M)$ lies on a simple closed curve. Then $\tilde f$ is one-to-one by Lemma \ref{lem:preX}, and hence so is $f$.

\emph{(Part I: Constructions of $M$ and $\tilde f$.)}
Let $C\subset\R^2$ be a  circle which encloses $\ol D$ and is disjoint from it. Then each ray $R_i$ must intersect $C$ at some point. Let $q_i\in R_i$ be the first such point,  assuming that $R_i$ is oriented so that $\ol p_i$ is its initial point, see Figure \ref{fig:disk}.   
  \begin{figure}[h] 
   \centering
   \begin{overpic}[height=1.1in]{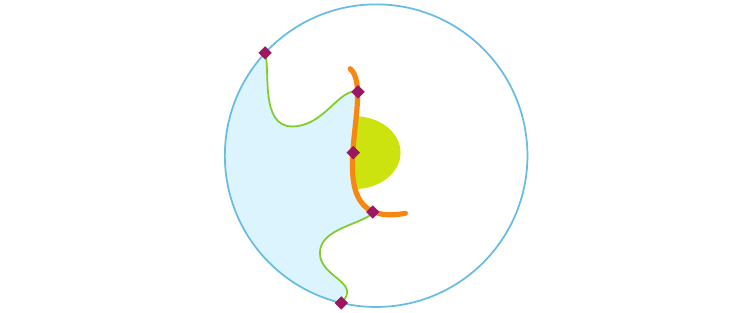} 
   \put(33,15){$D_i$}
   \put(43,20){$\ol x$}
   \put(50,18){$\ol W$}
    \put(33,38){$q_i$}
     \put(43,-3){$q_{i+1}$}
      \put(50,30){$\ol p_i$}
     \put(50,8.5){$\ol p_{i+1}$}
   \end{overpic}
   \caption{}
   \label{fig:disk}
\end{figure}
Now set
 $
 A_i:=\ol p_i q_i\cup\overline{p_i p_{i+1}}\cup\ol p_{i+1}q_{i+1}.
 $
  Then, by conditions (i) and (ii) of the theorem, $A_i$ is a simple curve. Consequently it divides  the disk bounded by $C$ into a pair of closed subdisks, which we call the sides of $A_i$.    
      Let $x\in\inte(p_ip_{i+1})$. By Lemma \ref{lem:X}, there is an open neighborhood $W$ of $x$ in $D$ such that 
     $\ol W-\overline{p_ip_{i+1}}$ is connected and is disjoint from $\overline{p_ip_{i+1}}$.  Further,  choosing $\epsilon$ sufficiently small in Lemma \ref{lem:X}, we can make sure that $\ol W$ is disjoint from $\ol p_i q_i$ and $\ol p_{i+1}q_{i+1}$. So it follows that $\ol W-A_i$ is connected and is disjoint from $A_i$.
     Consequently,  it lies in the interior of one of the sides of $A_i$. Let  $D_i$ be the opposite side.
       Glue each $D_i$ to $D_{i+1}$ along $\ol p_{i+1} q_{i+1}$. Further, glue each $D_i$ to $D$ by  identifying $p_i p_{i+1}$ with $\overline{p_i p_{i+1}}$ via $f$. This yields a compact connected surface $M$  which contains $D$. Define $\tilde f\colon M\to \R^2$ by letting $\tilde f= f$ on $D$, and  $\tilde f$ be the inclusion map $D_i\hookrightarrow\R^2$ on $D_i$. Then $\tilde f$ is continuous and $\tilde f(\partial M)\subset C$ as desired.

\emph{(Part II: Local injectivity of $\tilde f$.)}
Recall that $\tilde f$ is locally one-to-one on the interiors of $D$ and each $D_i$ by definition.  Also note that $\tilde f$ is one-to-one  
  near every point  of $C$ different from $q_i$. So 
   it remains to check that $\tilde f$ is  one-to-one near every point of  $\overline{p_ip_{i+1}}$ and $\ol p_{i}q_{i}$. There are four cases to consider:
 
\emph{(i)}
 First we check the points $\overline{x'}\in \inte(\overline{p_ip_{i+1}})$. It suffices to show that  there exists an open neighborhood $W'$ of $x'$ in $D$ such that $\overline{W'}-A_i$ is disjoint from $D_i$ (this would show that $D_i$ and $\ol D$ lie on different sides of $A_i$ near $\overline{x'}$).
 To see this let $X$ be the segment $xx'$ of $A_i$, and 
  $W'$ be a small open neighborhood of $X$ in $D$ given by Lemma \ref{lem:X}. Then,  just as we had argued earlier, $\overline{W'}-A_i$ will be disjoint from $A_i$, and thus will lie on one side of it. Since $x\in X$, $\overline{W'}-A_i$ intersects $\overline W-A_i$, which by definition lies outside $D_i$. Thus $\overline{W'}-A_i$ also lies outside $D_i$, as claimed.

\emph{(ii)} 
 Next we check $p_i$. Let $B:=B_\epsilon(\ol p_i)$, where $\epsilon>0$ is so small that $\ol p_{i-1}$, $\ol p_{i+1}$ and $q_i$ lie outside $B$. 
  Let $a$, $b$, $c$ be the first points where the (oriented curves) $\overline{p_i p_{i-1}}$, $\overline{p_i p_{i+1}}$, $\ol p_i q_i$ intersect $\partial B$ respectively.  Assuming $\epsilon$ is  small,  $ab$  will be simple, since $\overline{\partial D}$ is locally simple. Also note that $\ol p_i c$ is simple, since $R_i$  is simple. Furthermore, $\ol p_i c\cap ab\subset R_i\cap\overline{p_{i-1}p_{i+1}}=\{p_i\}$ by the second condition of the theorem. So  $ab\cup\ol p_i c$ divides $B$ into $3$ closed sectors, see the left diagram in Figure \ref{fig:disk2}. 
    \begin{figure}[h] 
   \centering
   \begin{overpic}[height=1in]{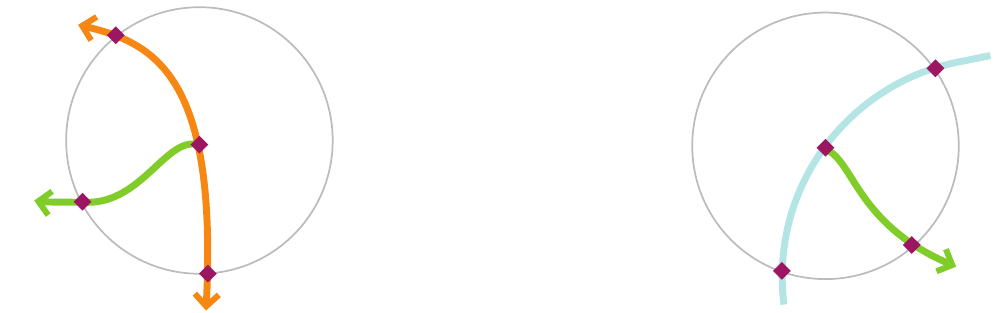} 
   \put(22,16){$\ol p_i$}
   \put(19,-3){$\ol p_{i+1}$}
   \put(-1,29){$\ol p_{i-1}$}
   \put(-1,10){$q_i$}
    \put(78,17){$q_i$}
    \put(97,2){$\ol p_{i}$}
     \put(11,29.5){$a$}
     \put(7,8){$c$}
     \put(22,3){$b$}
      \put(23,22){$S_1$}
      \put(10,18){$S_2$}
      \put(14,8){$S_3$}
      \put(100,24){$C$}
   \end{overpic}
   \caption{}
   \label{fig:disk2}
\end{figure}
   Let $S_1$ be the sector which contains $a$ and $b$, $S_2$ be the sector which contains $a$ and $c$, and $S_3$ be the sector which contains $c$ and $b$. 
 Next note that an open neighborhood of $p_i$ in $M$ consists of three components: 
a neighborhood $V_i$ of $p$ in $D$, and   neighborhoods $U_i$, $U_{i-1}$, of $\ol p_i$ in $D_i$, $D_{i-1}$ respectively. We claim that when these neighborhoods are small,  each lies in a different sector of $B$. Then, since $\tilde f$ is one-to-one on each of these neighborhoods, it will  follow that $\tilde f$ is one-to-one near $p_i$, as desired. To establish the claim note that by the third condition of the theorem there is an open neighborhood $V_i$ of $p_i$ in $D$ such that $\ol V_i$ is disjoint from the interior of $\ol p_i c$ (assuming $\epsilon$ is small). Further, we may assume that $V_i$ is connected and is so small that $\ol V_i$ fits inside  $B$. Then $\ol V_i$ must lie in $S_1$. Next, by Lemma \ref{lem:X}, we may choose a connected open neighborhood $U_i$ of $\ol p_i$ in $D_i$ such that $\ol U_i=U_i$ fits in $B$, and $\ol U_i -\overline{bc}=U_i-bc$ is connected, where $bc:=\ol p_i b\cup \ol p_i c$.  Note that  $U_i$  contains some interior points of $\ol p_i c$ and $\overline{p_i p_{i+1}}$. So $U_i-bc$ cannot lie entirely in $S_1$ or $S_2$, and therefore intersects $S_3$. Consequently $U_i-bc\subset S_3$, because $U_i-bc$ is connected and  disjoint from the boundary of $S_3$. So $U_i\subset S_3$. A similar argument shows that  $U_{i-1}\subset S_2$.

\emph{(iii)}
 Now we check the points $x'\in\inte(\ol p_i q_i)$. Let $X\subset \inte(\ol p_i q_i)$ be a connected compact set which contains $x'$ and a point of the neighborhood $U_{i-1}$ of $\ol p_i$  discussed in part (ii). Then again by Lemma \ref{lem:X}, there exists a connected open neighborhood $W'$ of $X$ in $D_{i-1}$ such that $W'-A_i$ lies entirely on one side of $A_i$. By design $W'-A_i$ intersects $U_{i-1}-A_i$, which lies outside $D_i$ as we showed in part (ii). Thus $W'-A_i$ also lies outside $D_i$. So $D_i$, $D_{i-1}$ lie on opposite sides of $A_i$ near $x'$, which shows that $\tilde f$ is one-to-one near $x'$.

\emph{(iv)}
 It remains to check $q_i$. Again, we have to show that there exists an open neighborhood  of $q_i$ in $D_{i-1}$ which lies outside $D_i$. The argument  is similar to that of part (ii), and uses part (iii).  Let $B:=B_\epsilon(q_i)$, where $\epsilon>0$ is so small that $B$  intersects $C$ in precisely two points and $\ol p_i$ lies outside $B$, see the right diagram in Figure \ref{fig:disk2}. Then the segment of $C$ in $B$ together with the  smallest segment of $q_i\ol p_i$  in $B$ determine three sectors. Only two of these sectors border both $C$ and a neighborhood of $q_i$ in $q_i\ol p_i$, and these are where $D_i$ and $D_{i-1}$ lie near $q_i$. We have to show that, near $q_i$, $D_i$ and $D_{i-1}$ lie in different sectors. To this end it suffices to note that every open neighborhood of $q_i$ in $D_{i-1}$, given by Lemma \ref{lem:X}, intersects a neighborhood of the type $W'$ discussed in part (iii), which lies outside $D_i$.
\end{proof}

\section*{Appendix B: Index of Symbols}

\bigskip

\begin{tabular}{l l l}
Symbol \hspace*{0.5in}& Principal Use \hspace*{2.75in}& Section\\
 & & \\
$P$&  a convex polyhedron  &  \ref{sec:intro}\\
$T$ &  a cut tree of $P$  &  \ref{sec:intro}\\
$P_T$ &  the compact disk obtained by cutting $P$ along $T$ &  \ref{sec:trees}\\
$h$ &  the height function &  \ref{subsec:term}\\
$\lambda$ & the stretching factor &  \ref{sec:intro}\\
$\pi\colon P_T\to P$ &  the natural projection &  \ref{sec:trees}\\
$\ol P_T$ & image of $P_T$ under an unfolding &  \ref{subsec:term}\\
$\Gamma=\Gamma_T$  &  the tracing path of $T$ &  \ref{sec:trees}\\
$\ol\Gamma$ &  image of $\Gamma$ under a (left) development &  \ref{sec:mixed}\\
$(\ol\Gamma)_{\gamma_i}$ &  a mixed development of $\Gamma$ based at the vertex $\gamma_i$ &  \ref{sec:mixed}\\ 
$\st_o$ & star of $P$ at a point $o$, &  \ref{subsec:sidesAndangles}\\
$\tilde \st_{\tilde o}$ & star of $P_T$ at a point $\tilde o$ &  \ref{sec:trees}\\
$[\gamma_0,\dots,\gamma_k]$ & a path with vertices $\gamma_i$ &  \ref{subsec:paths}\\
$\bullet$ & the operation for concatenation of two paths &  \ref{subsec:paths}\\
$\circ$ & the operation for  composition of two paths &  \ref{subsec:paths}\\
$\Gamma^{-1}$ & inverse of a path $\Gamma$ &  \ref{subsec:paths}\\
$\angle_P(o)$ & total angle of $P$ at a point $o$ &  \ref{subsec:sidesAndangles}\\
$\angle (a, o, b)$ &  (left) angle of the path $[a,o,b]$ at $o$ &  \ref{subsec:sidesAndangles}\\
$\theta_i$,  &{left angles of $\Gamma$} &  \ref{sec:mixed}\\
$\theta_i'$,  &{right angles of $\Gamma$} &  \ref{sec:mixed}\\
$v_i$ & vertices of $\Gamma_T$ &  \ref{sec:trees}\\
$\ol v_i$ & vertices of $\ol\Gamma_T$ which correspond to $v_i$ &  \ref{sec:trees}\\
$\tilde v_i$ & vertices of $P_T$ which correspond to $v_i$ &  \ref{sec:trees}\\
$\ell_i$ & leaves of $T$ as ordered by $\Gamma_T$ &  \ref{sec:monotonetrees}\\
$\ell_0$ & the top leaf of $T$ &  \ref{subsec:term}\\
$r$ & the  root of $T$ &  \ref{subsec:term}\\
$j_i$ & junctures of  $\Gamma_T$ &  \ref{sec:monotonetrees}\\
$\beta _i$ & branches of $T$ &  \ref{sec:monotonetrees}\\
$\beta _i'$ &  dual branches of $T$ &  \ref{sec:monotonetrees}\\
$\Gamma_i$ &  concatenation of the subpath $\ell_0\ell_i$ of $\Gamma_T$ with $\beta_i$ &  \ref{sec:monotonetrees}\\
$\Gamma_i'$ & concatenation of the subpath $\ell_0\ell_i$ of $\Gamma_T$ with $\beta_i'$ &  \ref{sec:monotonetrees}\\
$\tilde\Gamma_i'$ & the closed path in $P_T$ corresponding to $\Gamma_i'$ & \ref{sec:monotonetrees}\\
$D_i$ & the sub disk of $P_T$ bounded by $\tilde\Gamma_i'$ & \ref{sec:monotonetrees}\\
$D\Gamma$ & doubling of a path $\Gamma$ & \ref{sec:affine}\\

\end{tabular}

\newpage

\section*{Acknowledgement}
The author thanks A. J. Friend for computer experiments to test  Theorem \ref{thm:main} in the early stages of this work.

\bibliographystyle{abbrv}
\bibliography{references}

\end{document}